\documentclass[11pt,reqno,tbtags]{amsart}
\pdfoutput=1
\usepackage[]{amsmath,amssymb,amsfonts,latexsym,amsthm,enumerate}
\usepackage{amssymb}
\usepackage{enumerate,graphicx,subfigure}
\usepackage[numeric,initials,nobysame]{amsrefs}

\numberwithin{equation}{section}

\newtheorem{maintheorem}{Theorem}
\newtheorem{maincoro}[maintheorem]{Corollary}
\newtheorem{theorem}{Theorem}[section]
\newtheorem*{theorem*}{Theorem}
\newtheorem{conjecture}[theorem]{Conjecture}
\newtheorem*{conjecture*}{Conjecture}
\newtheorem{lemma}[theorem]{Lemma}
\newtheorem{claim}[theorem]{Claim}
\newtheorem{proposition}[theorem]{Proposition}
\newtheorem{observation}[theorem]{Observation}

\theoremstyle{definition}{

\newtheorem{definition}[theorem]{Definition}
\newtheorem*{definition*}{Definition}

}

\theoremstyle{remark}{

\newtheorem*{remark*}{Remark}

}

\newcommand{\deq}{\stackrel{\scriptscriptstyle\triangle}{=}}

\newcommand{\E}{\mathbb{E}}
\renewcommand{\P}{\mathbb{P}}
\DeclareMathOperator{\var}{Var}

\newcommand{\whp}{\ensuremath{\text{\bf whp}}}
\newcommand{\gap}{\text{\tt{gap}}}
\newcommand{\tmix}{t_\textsc{mix}}

\newcommand{\Po}{\operatorname{Po}}
\newcommand{\Bin}{\operatorname{Bin}}
\newcommand{\given}{\, \big| \,}
\newcommand{\one}{\boldsymbol{1}}
\newcommand{\Simple}{\ensuremath{\text{\sc Simple}}}

\renewcommand{\epsilon}{\varepsilon}
\renewcommand{\phi}{\varphi}

\newcommand{\G}{\mathcal{G}}
\newcommand{\cS}{\mathcal{S}}
\newcommand{\cP}{\mathcal{N}}
\newcommand{\cB}{\mathcal{B}}
\newcommand{\cT}{\mathcal{T}}
\newcommand{\cF}{\mathcal{F}}
\newcommand{\vW}{\overline{W}}
\newcommand{\tx}{\text{\tt{tx}}}
\newcommand{\vE}{\bar{E}}
\DeclareMathOperator{\dist}{dist}
\newcommand{\dir}[1]{\bar{#1}}
\newcommand{\dx}{{\dir{x}}}
\newcommand{\dy}{{\dir{y}}}
\newcommand{\SRW}{\textsf{SRW}}
\newcommand{\NBRW}{\textsf{NBRW}}
\newcommand{\droot}{\textsc{Lr}}

\date{}

\begin{document}
\title[Cutoff for random walks on random regular graphs]{Cutoff phenomena for random walks on random regular graphs}

\author{Eyal Lubetzky}
\address{Eyal Lubetzky\hfill\break
Microsoft Research\\
One Microsoft Way\\
Redmond, WA 98052-6399, USA.}
\email{eyal@microsoft.com}
\urladdr{}

\author{Allan Sly}
\address{Allan Sly\hfill\break
Microsoft Research\\
One Microsoft Way\\
Redmond, WA 98052-6399, USA.}
\email{allansly@microsoft.com}
\urladdr{}

\begin{abstract}
The cutoff phenomenon describes a sharp transition in the convergence of a family of ergodic finite Markov chains
to equilibrium. Many natural families of chains are believed to exhibit cutoff, and yet establishing this fact is often extremely challenging. An important such family of chains is the random walk on $\G(n,d)$, a random $d$-regular graph on $n$ vertices. It is well known that almost every such graph for $d\geq 3$ is an expander, and even essentially Ramanujan,
implying a mixing-time of $O(\log n)$. According to a conjecture of Peres,
the simple random walk on $\G(n,d)$ for such $d$ should then exhibit cutoff \whp. As a special case of this, Durrett conjectured that the mixing time of the lazy random walk on a random $3$-regular graph is \whp\ $(6+o(1))\log_2 n$.

In this work we confirm the above conjectures, and establish cutoff in total-variation, its location and its optimal window,
both for simple and for non-backtracking random walks on $\G(n,d)$. Namely, for any fixed $d\geq3$, the \emph{simple} random walk on $\G(n,d)$ \whp\ has cutoff at $\frac{d}{d-2}\log_{d-1} n$ with window order $\sqrt{\log n}$.
Surprisingly, the \emph{non-backtracking} random walk on $\G(n,d)$ \whp\ has cutoff already at $\log_{d-1} n$ with \emph{constant} window order.
We further extend these results to $\mathcal{G}(n,d)$ for any $d=n^{o(1)}$ that grows with $n$ (beyond which the mixing time is $O(1)$),
where we establish concentration of the mixing time on one of two consecutive integers.
\end{abstract}

\maketitle

\vspace{-0.4in}
\begin{figure}[h]
\centering
\includegraphics[width=1.5in]{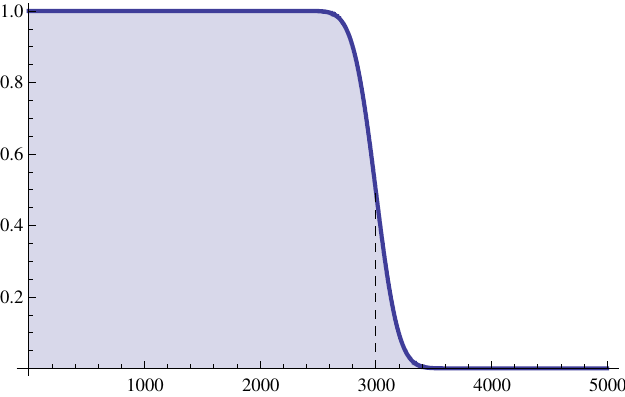}
\hspace{0.1in}
\includegraphics[width=1.5in]{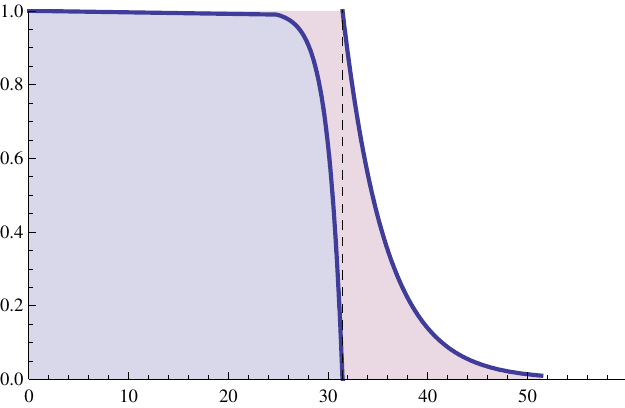}
\end{figure}
\vspace{-0.2in}

\section{Introduction}\label{sec:intro}

A finite ergodic Markov chain is said to exhibit \emph{cutoff} if its distance from the stationary measure drops abruptly, over a negligible time period known as the \emph{cutoff window}, from near its maximum to near $0$. That is, one has to run the Markov chain until the cutoff point in order for it to even slightly mix, and yet running it any further would be essentially redundant.

Let $(X_t)$ be an aperiodic irreducible Markov chain on a finite state space $\Omega$ with transition kernel $P(x,y)$ and stationary distribution $\pi$. The worst-case total-variation distance to stationarity at time $t$ is defined by
$$ d(t) \deq \max_{x \in \Omega} \| \P_x(X_t \in \cdot)- \pi\|_\mathrm{TV}~,$$
where $\P_x$ denotes the probability given $X_0=x$, and where $\|\mu-\nu\|_\mathrm{TV}$, the \emph{total-variation distance} of two distributions $\mu,\nu$ on $\Omega$, is given by
$$\|\mu-\nu\|_\mathrm{TV} \deq \sup_{A \subset\Omega} \left|\mu(A) - \nu(A)\right| = \frac{1}{2}\sum_{x\in\Omega} |\mu(x)-\nu(x)|~.$$
We define $\tmix(\epsilon)$, the total-variation \emph{mixing-time} of $(X_t)$ for $0 < \epsilon < 1$, as
$$ \tmix(\epsilon) \deq \min\left\{t : d(t) < \epsilon \right\}~.$$
Next, let $(X_t^{(n)})$ be a family of such chains, each with its corresponding worst-case total-variation distance from stationarity $d_n(t)$, its mixing-times $\tmix^{(n)}$, etc. We say that this family of chains exhibits \emph{cutoff} at time $\tmix^{(n)}(\frac{1}{4})$ iff the following sharp transition in its convergence to stationarity occurs:
\begin{equation}\label{eq-cutoff-def}\lim_{n\to\infty} \tmix^{(n)}(\epsilon) \big/ \tmix^{(n)}(1-\epsilon)=1 \quad\mbox{ for any $0 < \epsilon < 1$}~.\end{equation}
The rate of convergence in \eqref{eq-cutoff-def} is addressed by the following: A sequence $w_n=o\big(\tmix^{(n)}(\frac{1}{4})\big)$ is called a \emph{cutoff window} for the family of chains $(X_t^{(n)})$ if  for any $\epsilon > 0$ there exists some $c_\epsilon > 0$ such that for all $n$,
\begin{equation}\label{eq-window-def}\tmix^{(n)}(\epsilon) - \tmix^{(n)}(1-\epsilon) \leq c_\epsilon w_n~.\end{equation}
That is, there is cutoff at time $t_n=\tmix^{(n)}(\frac14)$ with window $w_n$ if and only if
$$\tmix^{(n)}(s) = \left(1+O(w_n)\right)t_n = (1+o(1))t_n~\mbox{ for any fixed $0 < s < 1$}~,$$
or equivalently, cutoff at time $t_n$ with window $w_n$ occurs if and only if
$$\left\{\begin{array}
  {l}\lim_{\lambda\to\infty} \liminf_{n\to\infty} \;d_n(t_n - \lambda w_n) = 1~,\\
  \lim_{\lambda \to \infty} \limsup_{n\to\infty} d_n(t_n + \lambda w_n) = 0~.
\end{array}\right.$$

Although many natural families of chains are believed to exhibit cutoff, determining that cutoff occurs proves to be
an extremely challenging task even for fairly simple chains, as it often requires
the full understanding of the delicate behavior of these chains around the mixing threshold. Before reviewing some of the related work in this area, as well as the conjectures that our work addresses, we state a few of our main results.

The focus of this paper is on random walks on a random regular graph, namely on $G \sim \G(n,d)$, a graph uniformly distributed over the set of all $d$-regular graphs on $n$ vertices, for $d\geq 3$ and $n$ large. This important class of random graphs has been extensively studied, among other reasons due to the remarkable expansion properties of its typical instance. One useful implication of these expansion properties is the rapid mixing of the corresponding \emph{simple random walk} (\SRW), the chain whose states are the vertices of the graph, and moves at each step to a uniformly chosen neighbor. Namely, the \SRW\ on such a graph has a mixing time of $O(\log n)$ \emph{with high probability} (\whp), that is, with probability tending to $1$ as $n\to\infty$.

Our first result establishes both cutoff and its optimal window for the \SRW\ on a typical instance of $\G(n,d)$ for any $d\geq 3$ fixed. As we later describe, this settles conjectures of Durrett \cite{Durrett} and Peres \cite{Peres} in the affirmative.

\begin{figure}
\centering \includegraphics[width=3.5in]{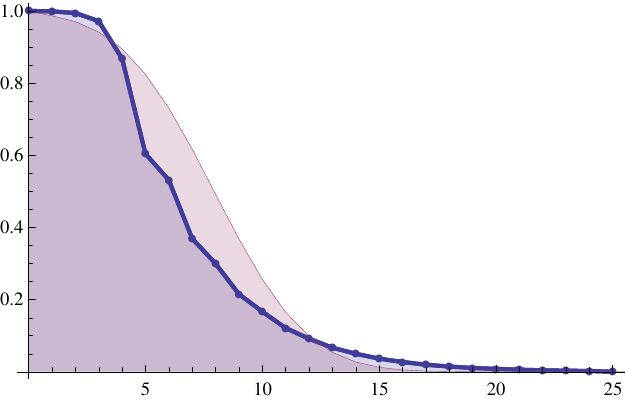}
\caption{Distance from stationarity
along time for the \SRW\ on a random $6$-regular graph on $n=5000$
vertices
.}
\label{fig:cutoff-rw}
\end{figure}

\begin{maintheorem}\label{thm-rw}Let $G \sim \mathcal{G}(n,d)$ be a random regular graph for $d \geq 3$ fixed. Then \whp, the \emph{simple}
random walk on $G$ exhibits cutoff at $\frac{d}{d-2}\log_{d-1} n$ with a window of order $\sqrt{\log n}$. Furthermore, for any fixed $0<s<1$, the worst case total-variation mixing time \whp\ satisfies
\begin{align*}
\tmix(s)  = \frac{d}{d-2}\log_{d-1} n - (\Lambda+o(1))\Phi^{-1}(s)  \sqrt{\log_{d-1} n}~,
\end{align*}
where $\Lambda=\frac{2\sqrt{d(d-1)}}{(d-2)^{3/2}}$ and $\Phi$ is the c.d.f. of the standard normal.
\end{maintheorem}
The essence of the cutoff for the \SRW\ on a typical $G\sim\G(n,d)$ lies in the behavior of its counterpart, the
non-backtracking random walk (\NBRW), that does not traverse the
same edge twice in a row (formally defined soon). Curiously, this chain also exhibits cutoff on $\G(n,d)$ \whp, only this time the cutoff window is \emph{constant}: \eqref{eq-window-def} holds for $w_n=1$ and $c_\epsilon$ logarithmic in $1/\epsilon$:
\begin{maintheorem}\label{thm-nbrw}Let $G \sim \mathcal{G}(n,d)$ be a random regular graph for $d \geq 3$ fixed. Then \whp, the \emph{non-backtracking}
random walk on $G$ has cutoff at $\log_{d-1}(dn)$ with a constant-size window. More precisely, for any fixed $\epsilon >0$, the worst case total-variation mixing time \whp\ satisfies
\begin{align*}
\tmix(1-\epsilon) &\geq \lceil\log_{d-1} (dn)\rceil - \lceil\log_{d-1}(1/\epsilon)\rceil ~,\\
\tmix(\epsilon) & \leq \lceil\log_{d-1} (dn)\rceil + 3\lceil\log_{d-1}(1/\epsilon)\rceil + 4~.
\end{align*}
\end{maintheorem}

\begin{figure}
\centering \includegraphics[width=3.5in]{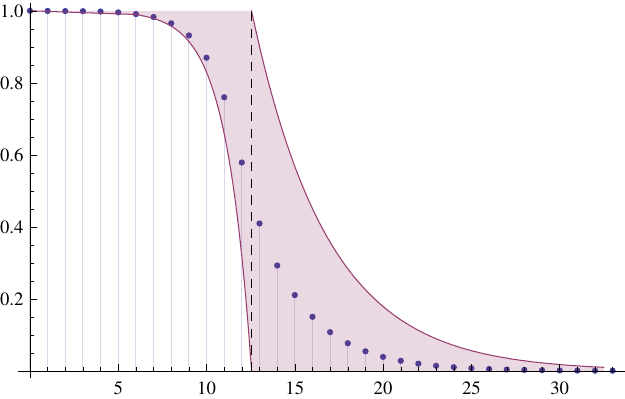}
\caption{Distance from stationarity
along time for the \NBRW\ on a random $3$-regular graph on $n=2000$
vertices
. Red curves represent
a ($4\log_{d-1}(1/\epsilon)$)-wide cutoff window.}
\label{fig:cutoff-nbrw}
\end{figure}

To gain insight to the above behaviors of the \SRW\ and \NBRW\ on a typical instance of $\G(n,d)$,
note that \whp, the random $d$-regular graph is locally-tree-like, its diameter is $(1+o(1))\log_{d-1} n$ and this is also the distance between a typical pair of vertices. In a $d$-regular tree, the height of a \SRW, started at the root, is analogous to a biased $1$-dimensional random walk with speed $(d-2)/d$. Hence, the time it takes this walk to reach height $\log_{d-1} n$
is concentrated around $\frac{d}{d-2}\log_{d-1}n$ with a standard deviation of order $\sqrt{\log n}$.
Our results establish that at this time, the walk on $\G(n,d)$ is mixed. One of the keys to showing this is estimating the number of simple paths of length just beyond $\log_{d-1}n$ between most pairs of vertices (see Lemma~\ref{lem-k-roots-intersections} for a more precise statement).
In comparison, as the \NBRW\ started at the root of a tree is forbidden from backtracking up, it reaches height $\log_{d-1} n$ after precisely $\log_{d-1} n$ steps, hence the sharper cutoff window.


Establishing the above theorems requires a careful analysis of the local geometry around typical pairs of vertices, via a Poissonization argument. Namely, we show that the number of edges between certain neighborhoods of two prescribed vertices is roughly Poisson. Similar arguments then allow us to formulate analogous results for the case of regular graphs of high degree, that is, $\G(n,d)$ where $d$ is allowed to tend to $\infty$ with $n$, up to $n^{o(1)}$.

\subsection{Related work}
The cutoff phenomenon was first identified for the case of random transpositions on the symmetric group in \cite{DS}, and for the case of the riffle-shuffle and random walks on the hypercube in \cite{Aldous}. In their seminal paper \cite{AD} from 1985,
Aldous and Diaconis established cutoff (and coined the term) for the top-in-at-random card shuffling process. See \cite{Diaconis} and \cite{CS} for more on the cutoff phenomenon, as well as \cite{SaloffCoste} for a survey of this phenomenon for random walks on finite groups.

Unfortunately, there are relatively few examples where cutoff has been rigorously shown, whereas many important chains are conjectured to exhibit cutoff. 
Indeed, merely deciding whether a given family of finite Markov chains exhibits cutoff or not (without pinpointing the precise cutoff location) can already be a formidable task (see \cite{Diaconis} for more on this problem).

In 2004, Peres \cite{Peres} proposed the condition $\tmix(\frac14) \cdot \gap \to\infty$ as a cutoff criterion, where $\gap$ is the spectral gap of the chain (i.e., $\gap\deq1-\lambda$ where $\lambda$ is the largest nontrivial eigenvalue of the transition kernel). While this ``product-condition'' is indeed necessary for cutoff in a family of reversible chains, there are known examples
where this condition holds yet there is no cutoff (see \cite{CS}*{Section 6}). Nevertheless, Peres conjectured that for many natural chains the product-condition does imply total-variation cutoff (e.g., this was recently verified in \cite{DLP} for the class of birth-and-death chains).

An important family of chains, mentioned in this context in \cite{Peres}, is \SRW s on transitive ``expander'' graphs of fixed degree $d$ (graphs where the second eigenvalue of the adjacency matrix is bounded away from $d$). Chen and Saloff-Coste \cite{CS} verified that such chains exhibit cutoff when measuring the convergence to equilibrium via other (less common) norms, and mentioned the remaining open problem of proving total-variation cutoff.

On the other hand, it is well known that almost every $d$-regular graph for $d\geq 3$ is an expander (see \cite{BS}, and also \cite{Pinsker} for an analogous statement under a closely related combinatorial definition of expansion). In fact, it was shown by Friedman \cite{Friedman} that the second eigenvalue of the adjacency matrix of $G\sim\G(n,d)$ for $d\geq 3$ is \whp\ $2\sqrt{d-1}+o(1)$, essentially as far from $d$ as possible. Thus, random regular graphs are a valuable tool for constructing sparse expander graphs, and furthermore, for any fixed $d\geq 3$, any
statement that holds \whp\ for $\G(n,d)$ also holds for almost every $d$-regular expander. See, \cite{Bollobas2},\cite{JLR} and also \cite{Wormald} for more on the thoroughly studied model $\G(n,d)$.

By the above, it follows that for any fixed $d\geq3$, the mixing time of the \SRW\ on $G \sim\G(n,d)$ is typically $O(\log n)$, whereas its $\gap$ is bounded away from $0$.
Hence, if we consider the \SRW\ on graphs $\{G_n\sim \G(n,d)\}$ for some fixed $d\geq3$, then the product-condition typically holds, and according to the above conjecture of Peres, these chains should exhibit cutoff \whp.

A special case of this was conjectured by Durrett, following his work with
Berestycki \cite{BD} studying the \SRW\ on a random $3$-regular graph $G\sim\G(n,3)$. They showed that at
time $c\log_2 n$ the distance of the walk from its starting point is asymptotically $(\frac{c}3\wedge 1)\log_2 n$. This implies a lower bound of $3\log_2 n$ for the asymptotic mixing time of random $3$-regular graphs, and in particular, an asymptotic lower bound of $6\log_2 n$ for the \emph{lazy} random walk (the lazy version of a chain with transition kernel $P$ is the chain whose transition kernel is $\frac12(P+I)$, i.e., in each step it stays in place with probability $\frac12$, and otherwise it follows the rule of the original chain). In \cite{Durrett}, Durrett conjectured that this latter bound is tight:

\begin{conjecture*}[Durrett \cite{Durrett}*{Conjecture 6.3.5}]
The mixing time for the lazy random walk on the random $3$-regular graph is asymptotically $6\log_2 n$.
\end{conjecture*}

Theorem \ref{thm-rw} stated above confirms these conjectures of Peres and Durrett (one can readily infer an upper bound on the mixing time of the lazy random walk from Theorem~\ref{thm-rw}).
Not only does this theorem establish cutoff and its location for the \SRW\ on $\G(n,d)$ (an analogous result immediately holds for the lazy walk), but it also determines the second order term in $\tmix(s)$ for any $0<s<1$ (the term corresponding to the cutoff window of order $\sqrt{\log n}$).

The \SRW\ on $\G(n,d)$ for $d=\lfloor(\log n)^a\rfloor$ and $a\geq 2$ fixed, starting from $v_1$ (not worst-case), was studied by Hildebrand \cite{Hildebrand}. He showed that
in this case there is cutoff \whp\ at $(1+o(1))\log_d n$,
and asked whether this also holds for $a < 2$. As we soon show, the answer to this question is positive, even from worst-case starting point and after replacing the $o(1)$ by an additive $2$. To describe this result, we must first discuss the \NBRW\ in further detail.

\begin{figure}[t]
\centering
\subfigure[\SRW\ on $\mathcal{G}(2^{1000},3)$]
{
    \label{fig:cutoff-asymp:a}
    \includegraphics[width=2.1in]{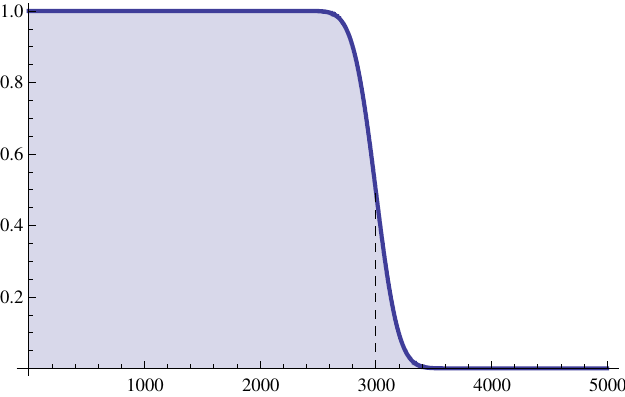}
}
\hspace{0.1in}
\subfigure[\NBRW\ on $\mathcal{G}(10^9,3)$]
{
    \label{fig:cutoff-asymp:b}
    \includegraphics[width=2.1in]{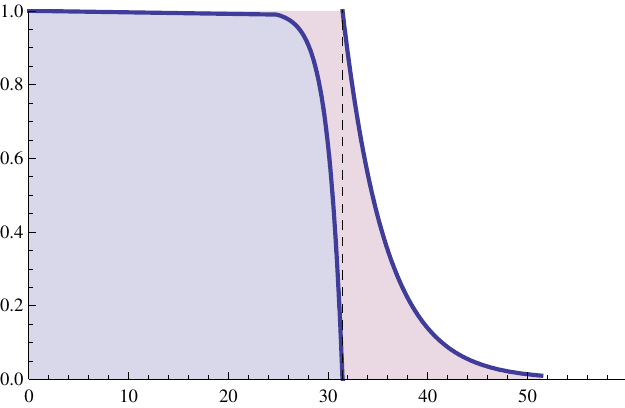}
}
\caption{Estimates on the total-variation distance from stationarity for \SRW s and \NBRW s
on large $3$-regular graphs.
(a) Asymptotic behavior of $\tmix$ established by Theorem \ref{thm-rw}.
(b) Lower and upper bounds according to Theorem \ref{thm-nbrw}.}
\label{fig:cutoff-asymp}
\end{figure}

\subsection{Cutoff for the \SRW\ and \NBRW}

While the \SRW\ of a graph is a Markov chain on its vertices, the \NBRW\ has
the set of directed edges (i.e., each edge appears in both orientations) as its state space: it moves from an edge $(x,y)$ to a uniformly chosen edge $(y,z)$ with $z\neq x$.
However, in most applications for \NBRW s on regular graphs (see, e.g., \cite{AFH} and the references therein), one often considers the projection of this chain onto the currently visited vertex (i.e., $(x,y)\mapsto y$), as it also converges to the uniform distribution on the vertices, and can thus be compared to the \SRW.

In \cite{ABLS} the authors compare the \SRW\ and this projection of the \NBRW\ on regular expander graphs, showing that the \NBRW\ has a faster \emph{mixing rate} (see \cite{Lovasz}
for the definition of this spectral parameter, which for the \SRW\ coincides with the largest nontrivial eigenvalue in absolute value). However, it was not clear how this spectral data actually translates into a direct comparison of the corresponding mixing times.

Theorems \ref{thm-rw} and \ref{thm-nbrw}, as a bi-product, enable us to directly compare the mixing times of the \SRW\ and \NBRW\ (not only its projection onto the vertices). Namely, we obtain that the \NBRW\ indeed \emph{mixes faster} than the \SRW\ on almost every $d$-regular graph, by a factor of $d/(d-2)$. Surprisingly, the delicate result stated in Theorem \ref{thm-nbrw} also shows that once we omit the ``noise'' created by the backtracking possibility of the \SRW, we are able to pinpoint the cutoff location up to $O(1)$
(see \cite{Hildebrand2} for an example of such an $O(1)$ cutoff window related to random walks on the symmetric group).

Recalling that the cutoff window in Theorem \ref{thm-nbrw} had the form $\log_{d-1} (1/\epsilon)$, one may wonder what the effect of large degrees would be. Our results extend to the case of large $d$, all the way up to $d=n^{o(1)}$, beyond which the mixing time is constant (see, e.g., \cite{Dou}) hence there is no point in discussing cutoff. The cutoff window indeed \emph{vanishes} as $d\to\infty$, and the entire mixing transition occurs within merely two steps of the chain:

\begin{maintheorem}\label{thm-nbrw-large-d}
Let $G \sim \mathcal{G}(n,d)$ where $d=n^{o(1)}$ tends to $\infty$ with $n$. Then \whp, for any fixed $0 < s < 1$,
the worst case total-variation mixing time of the \emph{non-backtracking} random walk on $G$ \whp\ satisfies
\begin{align*}
\tmix(s) \in \left\{ \lceil \log_{d-1} (dn) \rceil, \lceil \log_{d-1} (dn) \rceil+1\right\} ~.
\end{align*}
That is, the \emph{\NBRW}\ on $G$ has cutoff \whp\ within two steps of the chain.
\end{maintheorem}

As a corollary, the relation between \NBRW s and \SRW s directly implies an analogous statement for the \SRW\ on regular graphs of large degree. Here, the cutoff window
becomes $\sqrt{(1/d)\log_d n}$ (compared to $\sqrt{\log n}$ for $d$ fixed),
and if $\frac{\log n}{\log\log n}=o(d)$ then the walk completely coincides with the \NBRW.
\begin{maincoro}\label{cor-rw-large-d}
Let $G \sim \mathcal{G}(n,d)$ where $d=n^{o(1)}$ tends to $\infty$ with $n$. Then \whp,
the \emph{\SRW}\ on $G$ has cutoff at $\frac{d}{d-2}\log_{d-1}n$ with a window of
$\sqrt{\frac{\log n}{d\log d}}$. Furthermore, if $\frac{d\log\log n}{\log n}\to\infty$, then
for any fixed $0 < s < 1$,
the worst case total-variation mixing time of the \emph{\SRW} on $G$ \whp\ satisfies
\begin{align*}
\tmix(s) \in \left\{ \lceil \log_{d-1} (dn) \rceil, \lceil \log_{d-1} (dn) \rceil+1\right\} ~.
\end{align*}
\end{maincoro}

In particular, this answers the above question of Hildebrand (the case of $d =\lfloor (\log n)^a\rfloor$ for any $a > 0$ fixed) in the
affirmative, even from a worst starting position. Furthermore, instead of a multiplicative $1+o(1)$, the cutoff point is determined up to an additive 2 if $a\geq 1$.

\subsection{Random walks on the hypercube}
As mentioned above, one of the original examples of cutoff was for the lazy random walk on the hypercube $Q_m$,
which was shown by Aldous \cite{Aldous} to exhibit cutoff at $\frac12 m\log m$.
When compared to the \SRW\ on $\G(2^m,m)$, guaranteed by Corollary~\ref{cor-rw-large-d} to have cutoff \whp\ at $(\log 2+o(1)) m/\log m$
(in this setting, $d=\log_2 n$ has $\frac{d \log\log n}{\log n}\to\infty$), this demonstrates the slower than typical mixing of the hypercube.


\subsection{Organization}
The rest of the paper is organized as follows. Section~\ref{sec:prelim}
contains several preliminary facts on random regular graphs. In Sections~\ref{sec:rw} and \ref{sec:nbrw} we prove the main theorems, Theorems~\ref{thm-rw} and \ref{thm-nbrw}
resp., and in Section~\ref{sec:large-d} we extend these proofs to the case of $d$ large.

\section{Preliminaries}\label{sec:prelim}

Let $G=(V,E)$, and let $\vE$ denote the set of directed edges (i.e., $\vE$ contains both orientations of every edge in $E$). Throughout the paper, we will use $x,y,\ldots$ for vertices in $V$, as opposed to $\dx,\dy,\ldots$ for directed edges in $\vE$.

\subsection{The configuration model} This model, introduced by Bollob\'as \cite{Bollobas1} and sometimes also referred to as the \emph{pairing model}, provides
a convenient method of both constructing and analyzing a random regular graph. We next briefly review some of the properties of this model which we will need for our arguments (see \cite{Bollobas2},\cite{JLR} and \cite{Wormald}*{Section 2} for further information).

Given $d$ and $n$ with $dn$ even, a $d$-regular (multi-)graph on $n$ vertices is constructed via the configuration model as follows. Each vertex is identified with $d$ distinct points, and a random perfect matching of all these $d n$ points is then produced. The resulting multi-graph is obtained by collapsing every $d$-tuple into its corresponding vertex (possibly introducing loops or multiple edges). Let \Simple\ denote the event that the outcome is a simple graph.

It can easily be verified that, on the event \Simple, the resulting graph is uniformly distributed over $\G(n,d)$. Crucially, for any fixed $d$,
\begin{equation}\label{eq-p-simple}
  \P(\Simple) = (1+o(1))\exp\Big(\frac{1-d^2}4\Big)~,
\end{equation}
where the $o(1)$-term tends to $0$ as $n\to\infty$. In particular, as this probability is uniformly bounded away from $0$, any event that holds \whp\ for multi-graphs constructed via the configuration model, also holds \whp\ for $\G(n,d)$.

In fact, the statement in equation \eqref{eq-p-simple} was extended to any $d=o(n^{1/3})$ by McKay \cite{McKay}. Although the asymptotical behavior of this probability was thereafter determined for even larger values of $d$ (see \cite{Wormald} for additional information), in this work we are only concerned with the case $d=n^{o(1)}$, and hence this result will suffice for our purposes.

A highly useful property of the configuration model is the following: we can expose
the ``pairings'' sequentially, that is, given a vertex, we reveal the $d$ neighbors of its corresponding points one by one, and so on. This allows us to ``explore our way'' into the graph, while constantly maintaining the uniform distribution over the pairings of the remaining unmatched points.

\subsection{Neighborhoods and tree excess}
We need the following definitions with respect to a given graph $G=(V,E)$.
Let $\dist(u,v)=\dist_G(u,v)$ denote the distance between two vertices $u,v\in V$ in this graph.
 For any vertex $u\in V$ and integer $t$, the \emph{$t$-radius neighborhood} of $u$, denoted by $B_t(u)$, and its (vertex) boundary $\partial B_t(u)$, are defined as
 \begin{equation}
   \label{eq-t-radius-neighborhood}
   B_t(u) \deq \{ v\in V: \dist(u,v) \leq t\}~,~\partial B_t(u) \deq  B_t(u) \setminus B_{t-1}(u)~.
 \end{equation}
The abbreviated form $B_t$ will be used whenever the identity of $u$ becomes clear from the context. The \emph{tree excess} of $B_t$, denoted by $\tx(B_t)$, is the maximum number of edges that can be deleted from the induced subgraph on $B_t$ while keeping it connected (i.e., the number of extra edges in that induced subgraph beyond $|B_t|-1$).

The next lemma demonstrates the well known locally-tree-like properties of a typical $G\sim\G(n,d)$ for any fixed $d\geq3$. Its proof follows from a standard and straightforward application of the above mentioned ``exploration process'' for the configuration model.
\begin{lemma}\label{lem-leq-one-bad-edge}
Let $G\sim\G(n,d)$ for some fixed $d\geq3$, and let $t = \lfloor \frac{1}{5}\log_{d-1}n \rfloor$. Then \whp, $\tx(B_t(u)) \leq 1$ for all $u\in V(G)$.
\end{lemma}
\begin{proof}
Choose $u \in V$ uniformly at random, and consider the process where the neighborhood of $u$ is sequentially exposed level by level, according to the configuration model. When pairing the vertices of level $i$ (and establishing level $i+1$) for some $i \geq 0$, we are matching
$$m_i\leq d \;\vee\;(d-1)|\partial B_i| $$
points among a pool of $(1-o(1))dn$ yet unpaired points. For $1\leq k \leq m_i$, let $\mathcal{F}_{i,k}$ denote the $\sigma$-field generated by the process of sequentially exposing pairings up to the $k$-th unmatched point in $\partial B_i$. Further let $A_{i,k}$ denote the event that the newly exposed pair of the $k$-th unmatched point in $\partial B_i$ already belongs to some vertex in $B_{i+1}$. Clearly,
\begin{equation}  \label{eq-srw-bad-edge-i-k}
\P\left(A_{i,k} \mid \mathcal{F}_{i,k}\right) \leq \frac{(m_i-k) + (d-1)(k-1)}{(1-o(1))d n} \leq \frac{(d-1)m_i}{(1-o(1))dn} \leq\frac{m_i}{n}
\end{equation}
(where the last inequality holds for a sufficiently large $n$), and hence the number of events $\{A_{i,k}: 1\leq k \leq m_i\}$
that occur is stochastically dominated by a binomial random variable with parameters $\Bin(m_i,m_i/n)$.
(We say that $\mu$ stochastically dominates $\nu$, denoted by $\mu\succeq\nu$, if $\int f d\mu \geq \int f d\nu$
for every bounded increasing function $f$.)
Moreover, since $m_i \leq d (d-1)^i$ for any $0 \leq i \leq t$, it follows that $\sum_{i=0}^{t-1} m_i \leq d (d-1)^t$, and the
number of occurrences in the entire set of events  $\{A_{i,k} : i < t\}$
can be stochastically dominated as follows:
\begin{equation}
  \label{eq-aik-stoch-dom}
  \sum_{i=0}^{t-1}\sum_{k=1}^{m_i} \one_{A_{i,k}} \preceq \Bin\left( d(d-1)^t, \frac{d(d-1)^{t-1}}{n}\right)~.
\end{equation}
 Notice that, by definition, the number of such events
  that occur is exactly the tree excess of $B_{t}(u)$. We thus obtain that
\begin{align*}  
\P(\tx(B_t) \geq 2) &\leq O\left( \binom{d(d-1)^t}{2} \frac{d^2(d-1)^{2(t-1)}}{n^2} \right) = O\left(n^{-6/5}\right) ~,
\end{align*}
where the last equality is by the assumption on $t$. Taking a union bound over all vertices $u\in V$ completes the proof.
\end{proof}
When proving cutoff for the \NBRW\ in Section \ref{sec:nbrw}, we will be dealing with directed edges rather than vertices. The $t$-radius neighborhood of a directed edge $\dx$, denoted by $B_t(\dx)$, and its boundary $\partial B_t(\dx)$, then consist of directed edges, and are defined analogously to \eqref{eq-t-radius-neighborhood} (with $\dist(\dx,\dy)$ measuring the shortest non-backtracking walk from $\dx$ to $\dy$; note that $\dist(\cdot,\cdot)$ is not necessarily symmetric).
The tree excess $\tx(B_t(\dx))$ in this case will refer to the undirected underlying graph induced on $B_t(\dx)$.

\subsection{The cover tree of a regular graph}
Let $G=(V,E)$ be a $d$-regular graph and $u\in V$ be some given vertex in $G$. The \emph{cover tree of $G$ at $u$} is a mapping $\phi:\cT\to V$, where $\cT$ is a $d$-regular tree with root $\rho$, and the following holds:
\begin{equation}
  \label{eq-cover-tree}
  \left\{\begin{array}
  {l}\phi(\rho) = u~,\\
  N_G(\phi(x)) = \{\phi(y): y\in N_\cT(x)\}~\mbox{for any $x \in \cT$}~,
\end{array}\right.
\end{equation}
where $N_H(u)=\{v\in V(H):\dist_H(u,v)=1\}$ (i.e., $\partial B_1(v)$ for the graph $H$). That is, the root of $\cT$ is mapped to $u$, and $\phi$ respects $1$-radius neighborhoods.

The following two simple facts will be useful later on. First,
there is a one-to-one correspondence between non-backtracking paths in $G$ starting from $u$ and non-backtracking paths in $\cT$ starting from $\rho$. Second, if $X_t$ is a simple random walk on $\cT$, then $\phi(X_t)$ is a simple random walk on $G$.

\section{Cutoff for the simple random walk}\label{sec:rw}

In this section, we prove Theorem \ref{thm-rw}, which establishes cutoff for the \SRW\ on a typical random $d$-regular graph for any fixed $d\geq3$. Throughout this section, let $d\geq 3$ be some fixed integer, and consider some $G\sim\G(n,d)$.

We need the following definition concerning the locally tree-like geometry. 

\begin{definition}[$K$-root]\label{def-K-root}
We say that a vertex  $u\in V$ is a \textbf{$K$-root} if and only if the induced subgraph on $B_K(u)$ is a tree, that is, $\tx(B_K(u))=0$.
\end{definition}
Recalling Lemma \ref{lem-leq-one-bad-edge}, \whp\ every vertex in our graph $G\sim\G(n,d)$
has a tree excess of at most $1$ in its $\lfloor\frac15\log_{d-1}n\rfloor$-radius neighborhood. The next simple lemma shows that in such a graph (in fact, a weaker assumption suffices), a ``burn-in'' period of $\Theta(\log\log n)$ steps allows the \SRW\ from the worst-case starting position to reposition itself in a typically ``nice'' vertex.

\begin{lemma}\label{lem-srw-to-good-roots}
Let $K = \lfloor\log_{d-1}\log n\rfloor$, and suppose that every $u\in V$ has $\tx(B_{5K}(u))\leq 1$. Then for any $u \in V$, the \emph{\SRW}\ of length $4K$ from $(u,v)$ ends at a $K$-root with probability $1-o(1)$. In particular, there are $n-o(n)$ vertices in $G$ that are $K$-roots.
\end{lemma}
\begin{proof}
If $\tx(B_{5K}(u))= 0$ then the induced subgraph on $B_{5K}$ is a tree and the result is immediate.

If $\tx(B_{5K}(u))= 1$ then the induced subgraph on $B_{5K}$ is cycle $C$, with disjoint trees rooted on each of its vertices. Let
$X_t$ denote the position of the random walk at time $t$, and let $\rho_t=\dist(X_t,C)$, that is, the length of the shortest path between $C$ and $X_t$ in $G$.

 If the random walk is on the cycle then in the next step it either leaves $C$ with probability $\frac{d-2}{d}$, or remains on $C$ with probability $\frac2{d}$. Alternatively, if the random walk is not on $C$, then it moves one step closer to $C$ with probability $\frac1{d}$ and  one step further away with probability $\frac{d-1}{d}$.  Either way,
 $$\E [ \rho_{t+1}-\rho_t \mid X_t] = \frac{d-2}{d}~.$$
 Therefore, $\rho_t-\frac{(d-2)t}{d}$ is a martingale, and the Azuma-Hoeffding inequality (cf., e.g., \cite{AS}) ensures that
$$
\P\left(\left|\rho_{4K}-\rho_0-\frac{4K(d-2)}{d}\right|>\frac{K}{3}\right) \leq \exp\left(\frac{-K}{72\left(1+\frac{d-2}{d}\right)^2} \right)=o(1)~.
$$
We deduce that, \whp, $\rho_{4K}\geq \frac{4K(d-2)}{d} - \frac{K}{3} \geq K$ and hence $X_{4K}$ is a $K$-root.

To obtain the statement on the number of $K$-roots in $G$, suppose we start from a uniformly chosen vertex.  Clearly, the random walk at time $4K$ is also uniform, thus the probability that a uniformly chosen vertex is not a $K$-root is $o(1)$, as required.
\end{proof}
The following lemma demonstrates the control over the local geometry around a $K$-root with $K=\Theta(\log\log n)$.

\begin{lemma}\label{lem-u-boundary-sizes}
Set $R=\lfloor \frac{4}{7}\log_{d-1}n\rfloor$ and $K = \lfloor\log_{d-1}\log n\rfloor$. With high probability,
every $K$-root $u$ satisfies $$|\partial B_t(u)| \geq (1-o(1)) d(d-1)^{t-1} \mbox{ for all $t < R$}~.$$
\end{lemma}
\begin{proof}
Let $u$ be a uniformly chosen vertex; expose its $K$-neighborhood, and assume that it is indeed a $K$-root. Following the notation from the proof of Lemma \ref{lem-leq-one-bad-edge} we let $A_{i,k}$ be the event that, in the process of sequentially matching points, the newly exposed pair of the $k$-th unmatched point in $\partial B_i$ belongs to a vertex already in $B_{i+1}$.  Further recall that, by \eqref{eq-srw-bad-edge-i-k} and the discussion thereafter, the number of events $\{A_{i,k} : 0\leq i < R \}$ that occur is stochastically dominated by
a binomial variable with parameters $\Bin\left(d(d-1)^{R}, \frac{d(d-1)^{R-1}}{n}\right)$. Since the expectation of this random variable is
$$d^2 (d-1)^{2R-1} / n \leq O\big(n^{1/7}\big)~,$$ the number of events $A_{i,k}$ with $0\leq i < R$ that occur is less than $n^{1/6}$ (with room to spare) with probability at least $1 - \exp(-\Omega(n^{1/6}))$.

Each event $A_{i,k}$ reduces the number of leaves in level $i+1$ by at most 2 and so reduces the number of leaves in level $t>i$ by at most $2(d-1)^{t-i-1}$ vertices.  It follows that for each $0\leq t < R$,
\begin{equation}\label{e:badVertexBound-1}
|\partial B_t| \geq d(d-1)^{t-1} - \sum_{i<t}\sum_k \one_{A_{i,k}} 2(d-1)^{t-i-1}~.
\end{equation}
Set $L = \lfloor \frac15\log_{d-1} n\rfloor$. As $u$ is a $K$-root, no events of the form $A_{i,k}$ with $i<K$ occur, and the number of events $A_{i,k}$ which occur with $i<L$ is exactly $\tx(B_{L}(u))$, giving
$$ \sum_{i < L} \sum_k \one_{A_{i,k}} 2(d-1)^{t-i-1} \leq
2(d-1)^{t-K-1}\tx(B_{L}(u))~.$$
Furthermore, by the above discussion on the number of events $\{A_{i,k}\}$ that occur, we deduce that
with probability at least $1-\exp(-\Omega(n^{1/6}))$
$$ \sum_{i =L}^{t-1} \sum_k \one_{A_{i,k}} 2(d-1)^{t-i-1} \leq
2(d-1)^{t-L-1} n^{1/6} = o\left((d-1)^t\right).$$
Plugging the above in \eqref{e:badVertexBound-1} we get that with probability $1-\exp(-\Omega(n^{1/6}))$,
\begin{equation}\label{e:badVertexBound-2}
|\partial B_t| \geq (1-o(1))d(d-1)^{t-1} - 2(d-1)^{t-K} \tx(B_{L}(u))~,
\end{equation}
and a union bound implies that \eqref{e:badVertexBound-2} holds for all $K$-roots $u$ and all $t < R$ except with probability $\exp(-\Omega(n^{1/6}))$.

Finally, Lemma \ref{lem-leq-one-bad-edge} asserts that \whp\ every $u$ satisfies $\tx(B_{L}(u)) \leq 1$. Hence, \whp, every $K$-root $u$ satisfies $|\partial B_t|\geq (1-o(1))d(d-1)^{t-1}$ for all $0\leq t \leq R$, as required.
\end{proof}

Let $\partial B_t^*(u)$ denote the set of vertices in $\partial B_t(u)$ with a single (simple) path of length $t$ to $u$. We next wish to establish an estimate for the typical number of such vertices, intersected with some other neighborhood $B_{t'}(v)$.

\begin{lemma}\label{lem-k-roots-non-intersections}
Let $K = \lfloor \log_{d-1}\log n\rfloor$ and $R = \lfloor\frac{4}{7}\log_{d-1}n\rfloor $.
With high probability, any two  $K$-roots $u$ and $v$ with $\dist(u,v)>2K$ satisfy
$$\left|\partial B_t^*(u) \setminus B_{t+1}(v)\right| = (1-o(1))d(d-1)^{t-1} ~\mbox{ for all $t< R-1$}~.$$
\end{lemma}

\begin{proof}
The proof follows the same arguments as the proof of Lemma \ref{lem-u-boundary-sizes}, except now we begin with two randomly chosen vertices $u,v$. Expose $B_K(u)$ and $B_K(v)$, at which point we may assume that both $u$ and $v$ are $K$-roots, and that $\dist(u,v)>2K$. Next, we sequentially expand the layers
$$\partial \widetilde{B}_i\deq\{w\in V:\dist(w,\{u,v\})=i\}\mbox{ for $K< i \leq R$}~.$$
By the above assumption on $u$ and $v$, we have $$|\partial \widetilde{B}_K|=2d(d-1)^{K-1}~.$$
Repeating essentially the same calculations as those appearing in the proof of Lemma~\ref{lem-u-boundary-sizes} now shows that with probability $1-\exp(-\Omega(n^{1/6}))$,
\begin{equation}
  \label{eq-Bt-uv-bound}
  |\partial \widetilde{B}_{t}|=(2-o(1))d(d-1)^{t-1}\mbox{ for all $t\leq R$}~,
\end{equation}
thus \whp, the above holds for all pairs of $K$-roots $u,v$ with $\dist(u,v)>2K$.

We claim that the statement of the lemma follows directly from \eqref{eq-Bt-uv-bound}. To see this,
assume that \eqref{eq-Bt-uv-bound} indeed holds for $u,v$ as above, and let $t < R-1$. Clearly, at most $d(d-1)^{t-1}$ of the vertices in $\partial \widetilde{B}_{t}$ belong to $\partial B_{t}(v)$, hence $$\left|\partial B_{t}(u) \setminus B_{t}(v)\right|=(1-o(1))d(d-1)^{t-1}~,$$ and similarly,
  $$\left|\partial B_{t+1}(v) \setminus B_{t+1}(u)\right|=(1-o(1))d(d-1)^{t}~.$$ Therefore,
 \begin{align*}
&\left|\partial B_{t}(u) \cap B_{t}(v)\right|=o\left(d(d-1)^{t-1}\right)~,\\
&\left|\partial B_{t+1}(v) \cap B_{t+1}(u)\right|=o\left(d(d-1)^{t}\right)~,
 \end{align*}
and altogether we obtain that
\begin{align*}
\left|\partial B_t(u)\cap B_{t+1}(v)\right| &\leq  \left|\partial B_{t}(u) \cap B_{t}(v)\right| + \left| B_{t}(u) \cap \partial B_{t+1}(v)\right| \\
&= o(d(d-1)^{t})~.
\end{align*}
Since there are at most $d(d-1)^t$ paths of length $t$ from $u$ to $\partial B_t(u)$, and since $\left|\partial B_t(u)\right|=(1-o(1))d(d-1)^{t-1}$, it then follows that $$\left|\partial B_t(u) \setminus \partial B_t^*(u)\right|=o(d(d-1)^{t-1})~.$$  We deduce that $\left|\partial B_t^*(u)\cap B_{t+1}(v)\right| =o(d(d-1)^{t})$,
and the proof follows.
\end{proof}

\begin{lemma}\label{lem-k-roots-intersections}
Let $K = \lfloor\log_{d-1}\log n\rfloor$ and $T = \lfloor \frac12 \log_{d-1}n\rfloor$. With high probability, any two  $K$-roots $u$ and $v$
with $\dist(u,v)>2K$ satisfy
$$\cS_{2T+\ell}(u,v) \geq (1-o(1))\frac1{n}d(d-1)^{2T+\ell-1}$$
for all $2K \leq \ell \leq \frac1{20}\log_{d-1}n$, where
$\cS_k(u,v)$ denotes the number of \emph{simple} paths of length $k$ between $u$ and $v$,
and the $o(1)$-term tends to $0$ as $n\to\infty$.
\end{lemma}
\begin{proof}
Fix $\ell$ as above and expose the neighborhoods of $u$ and $v$ up to distance
$$t_u=\left\lceil\mbox{$\frac{1}{2}$}(2T+\ell-1)\right\rceil~,~t_v=\left\lfloor\mbox{$\frac{1}{2}$}(2T+\ell-1)\right\rfloor$$ respectively. Notice that this selection gives
$$2T+\ell-1 =t_u+t_v~,~0\leq t_u-t_v\leq 1~.$$  We further define
$$A_u=\partial B_{t_u}^*(u) \setminus B_{t_v}(v)~,~A_v=\partial B_{t_v}^*(v) \setminus B_{t_u}(u)~.$$
We may now assume that the statement of Lemma \ref{lem-k-roots-non-intersections} holds with respect to the
neighborhoods of $u$ and $v$ already revealed (and them alone), that is
\begin{align*}
 \left|A_u\right|&=(1-o(1))d(d-1)^{t_u-1}~,\\
 \left|A_v\right|&=(1-o(1))d(d-1)^{t_v-1}~.
\end{align*}
In other words, $A_u$ has $(1-o(1))d(d-1)^{t_u}$ unmatched points and similarly, $A_v$ has $(1-o(1))d(d-1)^{t_v}$ unmatched points.

Now, sequentially match each of the points in $A_u$, and let $M_{u,v}$ denote the number of points of $A_u$ matched with points in $A_v$.
To obtain an upper bound on $M_{u,v}$, we once again repeat the arguments of Lemma \ref{lem-leq-one-bad-edge}, implying that it is stochastically bounded from above by a binomial
variable as follows
$$ M_{u,v} \preceq \Bin\Big((d-1)|A_u|, \frac{(d-1)|A_v|}{(1-o(1))dn}\Big)~.$$
Since
$$\frac{(d-1)^2|A_u||A_v|}{dn}\leq O(n^{1/10})~,$$ Chernoff bounds (cf., e.g., \cite{AS}) give that $M_{u,v} \leq n^{1/4}$ except with probability $\mathrm{e}^{-\Omega(n^{1/4})}$.  We thus assume that indeed $M_{u,v} \leq n^{1/4}$.

In this case, as we sequentially match points, each point in $A_u$ has at least $|A_v|-n^{1/4}$ remaining points in $A_v$ which it could potentially be matched to. That is, conditional on previous matchings each point has at least  $\frac{|A_v|-n^{1/4}}{dn}$ probability of being matched to a point in $A_v$.  It follows that $M_{u,v}$ is stochastically bounded from below by a binomial variable
 $$ M_{u,v} \succeq \Bin\Big((d-1)|A_u|, \frac{(d-1)(|A_v|-n^{1/4})}{dn}\Big)~.$$  Now $$\frac{(d-1)^2|A_u|(|A_v|-n^{1/4})}{dn} = (1-o(1))\frac1n d(d-1)^{2T+\ell-1}= \Omega(\log_{d-1}^2 n)~,$$ and again by Chernoff bounds we have that the number of matchings is at least $(1-o(1))\frac1n d(d-1)^{2T+\ell-1}$ except with probability $$\exp(-\Omega(\log_{d-1}^2 n))=o(n^{-3})~.$$  Each matching between a point in $A_u$ and a point in $A_v$ determines a simple path from $u$ to $v$ of length $2T+\ell$, thus $$\cS_{2T+\ell}(u,v) \geq M_{u,v} \geq (1-o(1))\frac1{n}d(d-1)^{2T+\ell-1}~.$$  Taking a union bound over all $u$, $v$ and $\ell$ completes the result.
\end{proof}

\begin{proof}[\emph{\textbf{Proof of Theorem \ref{thm-rw}}}]
Set $K = \lfloor \log_{d-1}\log n \rfloor$ and set $T = \lfloor \frac12 \log_{d-1}n\rfloor$.  By Lemma \ref{lem-srw-to-good-roots}, after $4K$ steps with high probability the random walk is at a $K$-root.  Since we are only seeking to establish $\tmix$ up to an accuracy of  $o(\sqrt{\log_{d-1} n})$ and since $K=o(\sqrt{\log_{d-1} n})$ it is enough to consider the worst case mixing from a $K$-root to establish the result.

Let us assume that the statement of Lemma \ref{lem-k-roots-intersections} holds.    Let $u$ and $v$ be $K$-roots with $\dist(u,v)> 2K$.  By Lemma \ref{lem-k-roots-intersections}, $$\cS_{2T+\ell}(u,v)\geq \frac{1-o(1)}n d(d-1)^{2T+\ell-1}~\mbox{ for  $2K \leq \ell \leq \frac1{20}\log_{d-1}n$}~.$$
Now let $\cT$ be the cover tree for $G$ at $u$ with a map $\phi$, as defined in \eqref{eq-cover-tree}.  Since each simple path in $G$ corresponds to a distinct simple path in $\cT$,
\begin{align*}
\#\left\{ w\in \cT: \phi(w)=v,\; \dist(\rho,w)= 2T+\ell \right\} &\geq \cS_{2T+\ell}(u,v) \\
&\geq \frac{1-o(1)}n d(d-1)^{2T+\ell-1}~,
\end{align*}
when  $2K \leq \ell \leq \frac1{20}\log_{d-1}n$.  Let $X_t$ be a \SRW\ on $\cT$ started from $\rho$ and let $W_t=\phi(X_t)$ be the corresponding \SRW\ on $\G$ started from $u$.  Note that, by symmetry, conditioned on $\dist(\rho,X_t)=k$ the random walk is uniform on the $d(d-1)^{k-1}$ points $\{w \in \cT: \dist(\rho,w)=k\}$.  In addition, a random walk on a $d$-regular tree with $d\geq 3$ is transient, since the distance from the root is a biased random walk with positive speed.  In particular, the random walk returns to $\rho$ only a finite number of times almost surely.  If $X_t \neq \rho$ then
$$ \big(\dist(X_{t+1},\rho) - \dist(X_{t},\rho)\big) \sim\left\{\begin{array}
  {ll}-1 & 1/d~,\\
  1 & (d-1)/d~.
\end{array}\right.$$
Therefore, the Central Limit Theorem gives that
\begin{equation}\label{e:clt}
\frac{\dist(X_{t},\rho) - \frac{(d-2)t}{d}}{\frac{2\sqrt{d-1}}{d}\sqrt{ t}} \stackrel{\mathrm{d}}{\longrightarrow} N(0,1).
\end{equation}
Let $A$ be the set of vertices which are $K$-roots and whose distance from $u$ is greater than $2K$.  Since there are at most $d(d-1)^{2K-1}=o(n)$ vertices within distance $2K$ of $u$, and since by Lemma \ref{lem-srw-to-good-roots} there are $n-o(n)$ $K$-roots in total, it follows that $|A| \geq n-o(n)$.

Combining these arguments, we deduce that if $v\in A$ and
\begin{equation}\label{eq-srw-t-choice}t=\Big\lfloor \frac{d}{d-2}\log_{d-1} n + k \sqrt{\log_{d-1}n}\Big\rfloor\end{equation} then
\begin{align*}
\P(W_t = v) &= \sum_{j=0}^t \P(\dist(\rho,X_t) = j) \frac{\#\{ w\in \cT: \phi(w)=v,\; \dist(\rho,w)=j \}}{d(d-1)^{j-1}}\\
&\geq \sum_{\ell=2K}^{\frac1{20}\log_{d-1}n} \P(\dist(\rho,X_t) = 2T+\ell) \frac{\frac{1+o(1)}n d(d-1)^{2T+\ell-1}}{d(d-1)^{2T+\ell-1}}\\
&= (1+o(1)\frac1n \P\left(2T + 2K \leq \dist(\rho,X_t) \leq 2T +  \frac1{20}\log_{d-1}n\right)\\
&= (1+o(1))\frac1n \left(1-\Phi\left( \frac{-k}{\Lambda} \right) \right)~,
\end{align*}
where the final equality follows from equation \eqref{e:clt} and where $\Phi$ is the distribution function of the standard normal and $\Lambda=\frac{2\sqrt{d-1}}{d-2}\sqrt{\frac{d}{d-2}}$.  Then
\begin{align}\label{e:tvUBound}
\| \P(W_t\in \cdot)&-\pi \|_{\mathrm{TV}}
=\sum_{v\in V} \max\left \{ \frac1n-\P(W_t =v)\;,\;0 \right \}\nonumber\\
&\leq \frac{n-|A|}{n} + \sum_{v\in A} \max\left \{\frac1n - \P(W_t=v)\;,\;0 \right \} \nonumber\\
&\leq o(1)+ (1+o(1)) |A| \frac1n \Phi\left( \frac{-k}{\Lambda} \right) = (1+o(1))  \Phi\left( \frac{-k}{\Lambda} \right)~.
\end{align}
It remains to provide a matching lower bound for $\| \P(W_t\in \cdot)-\pi \|_{\mathrm{TV}}$. To this end, let $M=\log_{d-1}n - K$ and note that $$\pi(B_M(u))\leq \frac1n d(d-1)^{M-1} =o(1)~.$$  If $w\in T$ and $\dist(\rho,w)\leq M$ then $\phi(w)\in B_M$.  For the same choice of $t$ as given in \eqref{eq-srw-t-choice}, equation \eqref{e:clt} gives that $$\P(\dist(X_t,\rho) \leq M) =    (1+o(1))  \Phi\left( \frac{-k}{\Lambda} \right)~,$$ and so $$\P(W_t \in B_M) \geq  (1+o(1))\Phi\left( \frac{-k}{\Lambda} \right)~.$$  It follows that
\begin{equation}\label{e:tvLBound}
\| \P(W_t\in \cdot)-\pi \|_{\mathrm{TV}} \geq \P(W_t\in B_M) - \pi(B_M) = (1+o(1))\Phi\left( \frac{-k}{\Lambda} \right).
\end{equation}
Combining equations \eqref{e:tvUBound} and \eqref{e:tvLBound} establishes that for any $0<s<1$
$$
\tmix(s)=\log_{d-1} n - (\Lambda+o(1)) \Phi^{-1}(s)\sqrt{\log_{d-1} n}~,
$$
completing the proof.
\end{proof}

\section{Cutoff for the non-backtracking random walk}\label{sec:nbrw}
In this section, we prove Theorem \ref{thm-nbrw} that establishes the cutoff of the \NBRW\ on a typical
random $d$-regular graph for $d\geq 3$ fixed. Throughout this section, let $d\geq 3$ be some fixed integer, and consider some $G\sim\G(n,d)$.

Since the \SRW\ induces a cutoff window of order $\sqrt{\log n}$ merely on account of its backtracking ability,
throughout our arguments in Section \ref{sec:rw} we could easily afford burn-in periods of order $\log\log n$.
On the other hand, our statements for the \NBRW\ establish a constant cutoff window (and moreover, logarithmic in $1/\epsilon$), and therefore require a far more delicate approach.

Recall that the \NBRW\ is a Markov chain on the set of directed edges; we thus begin by defining
a \emph{directed} $K$-root, analogous to Definition \ref{def-K-root}.

\begin{definition}[\emph{directed $K$-root}]\label{def-directed-k-root}
A directed edge $\dx\in\vE$ is a directed $K$-root iff the induced subgraph on $B_K(\dx)$ is a tree, i.e., $\tx(B_K(\dx))=0$.
\end{definition}

As before, it is straightforward to show that the directed edges of $G$ have locally-tree-like neighborhoods. This is stated by
the next lemma.
\begin{lemma}\label{lem-dir-leq-one-bad-edge}
Let $L = \lfloor \frac{1}{5}\log_{d-1}n \rfloor$. Then \whp, $\tx(B_L(\dx)) \leq 1$ for all $\dx\in \vE$. In addition, for any $r=r(n)$ and $h=h(n)\to\infty$ arbitrarily slowly, \whp\ at least $dn-h(d-1)^{2r}$ directed edges satisfy $\tx(B_r) = 0$.
\end{lemma}
\begin{proof}
Clearly, if $\dx=(u,v)\in\vE$ we have $\tx(B_t(\dx)) \leq \tx(B_t(v))$ for any $t$, thus the first statement of the lemma follows immediately from Lemma \ref{lem-leq-one-bad-edge}.

To show the second statement, recall the exploration process performed in the proof Lemma \ref{lem-leq-one-bad-edge}, where $A_{i,k}$ denoted
the event that the $k$-th matching generated in the $i$-th layer already belongs to our exposed neighborhood. In our setting, we perform a similar exploration process on a random $\dx=(u,v)\in\vE$, only this time the initial vertex $v$ corresponds to $d-1$ points rather than $d$ (having excluded its edge to $u$). Thus, \eqref{eq-aik-stoch-dom} translates into
\begin{equation*}
  \sum_{i=0}^{t-1}\sum_{k=1}^{m_i} \one_{A_{i,k}} \preceq \Bin\left( (d-1)^{t+1}, \frac{(d-1)^{t}}{n}\right)~.
\end{equation*}
It follows that the probability that $\tx(B_r(\dx)) > 0$ is at most $O\left(d-1)^{2r}/n\right)$,
and the expected number of such $\dx\in\vE$ is $O\left((d-1)^{2r}\right)$, as required.
\end{proof}
The following lemma, which is the analogue of Lemma \ref{lem-srw-to-good-roots}, shows that a small burn-in period typically brings
the \NBRW\ to a directed $L$-root for a certain $L$ (and allows us to restrict our attention to such starting positions).
\begin{lemma}\label{lem-nbrw-to-good-roots}
Let $\epsilon > 0$, set $K = \lceil \log_{d-1}(2/\epsilon) \rceil$ and $L=\lfloor \frac{1}{6}\log_{d-1}n \rfloor$. Let $\dx\in \vE$ be such that
$\tx(B_{K+L}(\dx))\leq 1$.
Then the non-backtracking walk of length $K$ from $\dx$ ends at a directed $L$-root with probability at least $1-\epsilon$.
\end{lemma}
\begin{proof}
Let $H$ be the subgraph formed by the elements (directed edges) of $B_{K+L}(\dx)$, and notice that the $L$-radius neighborhoods
of all possible endpoints $\dy$ of a non-backtracking walk of length $K$ from $\dx$ are all contained in $H$. Thus, if $\tx(B_{K+L}(\dx))=0$ then clearly every such endpoint is a directed $L$-root.

Otherwise, consider the undirected underlying graph of $H$. This graph contains a single simple cycle $C$ (by the assumption that
$\tx(B_{K+L}(\dx))\leq 1$), therefore the distance of any vertex $u\in H$ from $C$ is well defined.
Let $(\vW_t)$ denote the non-backtracking random walk started at $\vW_0 = \dx$.
For some $1 \leq t < K$, write $\vW_t=(u,v)$ and $\vW_{t+1}=(v,w)$.
Crucially, we claim that if $\dist(v,C)  < \dist(w,C)$, then $\vW_j$ is a directed $L$-root for all $j\in\{t+1,\ldots,K\}$.
Indeed, our subgraph consists of a cycle $C$ with disjoint trees rooted at some of its vertices. Therefore, as soon as the non-backtracking walk makes a single step away from $C$, by definition it can only traverse further away from $C$ with each additional step (as long as it is in $H$).

Furthermore, if $v \notin C$ (that is, $v$ belongs to one of the trees rooted on $C$), then
with probability $\frac{1}{d-1}$ the distance to $C$ decreases by $1$ in $\vW_{t+1}$, otherwise it increases by $1$.
Similarly, 
$$\P\left(w \in C \mid u,v\in C\right)= 1/(d-1)~.$$
The remaining case is the \emph{single} step immediately following the first visit to the cycle $C$, if such exists,
where the probability of remaining on $C$ (traversing along one of the two possible directions on it) is $\frac{2}{d-1}$.
Altogether,
$$ \P_\dx(\vW_{K}\mbox{ is not a directed $L$-root}) \leq 2(d-1)^{-K} \leq \epsilon,$$
as required.
\end{proof}

The next two lemmas are the analogues of Lemmas \ref{lem-u-boundary-sizes} and \ref{lem-k-roots-non-intersections} for directed $K$-roots, and both follow by essentially repeating the original arguments.
\begin{lemma}\label{lem-dir-root-boundary-sizes}
Set $T = \frac{51}{100}\log_{d-1}n$ and $K = K(n)$. Then with probability $1-o(n^{-3})$,
every directed $K$-root $\dx$ satisfies $$|\partial B_t(\dx)| \geq \left(1-(d-1)^{-K}-O(n^{-1/5})\right) (d-1)^t \mbox{ for all $t \leq T$}~.$$
\end{lemma}

\begin{lemma}\label{lem-dir-k-roots-intersections}
Let $\epsilon > 0$, $T = \frac{51}{100} \log_{d-1}n $ and $L = \lceil \frac{1}{6}\log_{d-1}n\rceil$.
With probability $1-o(n^{-3})$, any two directed $L$-roots $\dx$ and $\dy$ with $\dist(\dx,\dy) > 2L$ satisfy
$$\left|B_t(\dx) \cap B_t(\dy)\right| < n^{-1/7}(d-1)^t ~\mbox{ for all $t\leq T$}~.$$
\end{lemma}



We now turn to prove the Poissonization argument, on which the entire proof of Theorem \ref{thm-nbrw} hinges.
Recall that in Theorem \ref{thm-rw} we could afford a relatively large (order $\log\log n$) error, which enabled us to apply standard large deviation arguments for the size of cuts
between certain neighborhoods of two vertices $u,v$ (as studied in Lemma \ref{lem-k-roots-intersections}). On the other hand, here we can only afford an $O(1)$ error, so the number of paths of length the mixing time between two random vertices will approximately be a Poisson random variable with constant mean.
In order to bypass this obstacle and derive the concentration results needed for proving cutoff, we instead consider the joint distribution of $u$ and vertices $v_1,\ldots,v_M$ for some large (poly-logarithmic) $M$. This approach, incorporated in the next proposition, amplifies the error probabilities as required.
\begin{proposition}\label{prop-poisson}
Let $\epsilon > 0$, set
\begin{align*}
K = \lceil2\log_{d-1}(1/\epsilon)\rceil~, \quad T = \lceil\log_{d-1} (dn)\rceil ~, \quad\mu = (d-1)^{T+K}/dn~,
\end{align*}
and for each $\dx\in\vE$, define the random variable $Z=Z(\dx)$ by
\begin{align*}
  \P(Z = k) = \frac{1}{dn}\left|\left\{ \dy \in \vE : \cP_{T+K-1}(\dx,\dy) =k \right\} \right|~,
\end{align*}
where $\cP_\ell(\dx,\dy)$ is the number of $\ell$-long non-backtracking paths from $\dx$ to $\dy$.
Then \whp, every $\dx$ that is a directed $L$-root for $L=\lceil\frac16\log_{d-1}(dn)\rceil$ satisfies
$$\E \left[\left| (Z(\dx)/\mu)-1\right|\given \cF_G \right]<2\epsilon+\frac{5}{\log\log n}~,$$
where $\cF_G$ is the $\sigma$-field generated by the graph $G\sim \mathcal{G}(n,d)$.
\end{proposition}
\begin{proof}
Condition on the statement of Lemma \ref{lem-dir-leq-one-bad-edge} for the choices $r(n)=L$ and $h(n)=\log n$. That is,
we assume that there are at least $dn - (\log n) n^{1/3}$ directed $L$-roots in $\vE$.

Let $\dx$ be a uniformly chosen directed edge, and expose its $L$-radius neighborhood according to the configuration model. As the statement of the proposition only refers to directed $L$-roots, we may at this point assume that $\dx$ is indeed such an edge (recall that the property of being a directed $L$-root is solely determined by the structure of the induced subgraph on $B_L(\dx)$, and thus this conditioning does not affect the distribution of the future pairings). With this assumption in mind, continue exposing the neighborhood of $\dx$ to obtain $B_{2L}(\dx)$.

Our goal is to show that
$$\P\left(\E \left[\left|(Z(\dx)/\mu)-1\right|\given\cF_G\right]\geq2\epsilon+\mbox{$\frac{5}{\log\log n}$}\right) = o(1/n)~,$$
in which case a first moment argument will immediately complete the proof of the proposition.

We next consider a uniformly chosen set of $M$ directed edges, $\cB \subset \vE$,
for some $\log^2 n \leq M \leq 2\log^2 n$ (to be specified later), by selecting its elements one by one. That is, after $i$ steps ($0 \leq i <M$), $|\cB|=i$ and we add a directed edge uniformly chosen over the $dn-i$ remaining elements of $\vE$. With the addition of every new element, we also develop its $2L$-radius neighborhood.

Notice that, after $i$ steps, there are at most $(\log n) n^{1/3}$ directed edges which are \emph{not} directed $L$-roots in $\vE$, and furthermore,
$$ \left|B_{2L}(\dx) \cup \left(\cup_{\dy\in\cB}B_{2L}(\dy)\right)\right| \leq (i+1) n^{1/3} \leq M n^{1/3}~.$$
Therefore, the probability that the $(i+1)$-th element of $\cB$ either belongs to one of the existing $2L$-radius neighborhoods, or is \emph{not} a directed $L$-root, is at most $2 M n^{-2/3}$.
Clearly, the probability that $4$ such ``bad'' edges are selected is at most
$O(M^4 n^{-8/3}) = o(n^{-2})$.

Altogether, we may assume with probability $1-o(n^{-2})$, the set $\cB$ contains a subset $\cB'=\{\dy_1,\ldots,\dy_{M'}\}$ of size $M' \geq M-3$, such that the following holds:
\begin{enumerate}[(i)]
\item Every member of $\{\dx\} \cup \cB'$ is an $L$-root.
\item The pairwise distances of $\{\dx\}\cup \cB'$ all exceed $2L$.
\end{enumerate}

For any $\dy\in\vE$, let $Z_\dy = \cP_{T+K-1}(\dx,\dy)$, and for any $S\subset \vE$, let $Z_S$ be the random variable that accepts the value $Z_\dy$ with probability $1/|S|$ for each $\dy\in S$. We will use an averaging argument to show that $Z$ can be well approximated by $Z_\cB$, which in turn is well approximated by $Z_{\cB'}$.

Setting
\begin{equation*}
  T_1 = \lfloor (T+K)/2 \rfloor ~,~ T_2 = \lceil (T+K)/2 \rceil - 2 ~,~
\end{equation*}
we wish to develop the $T_1$-radius neighborhood of $\dx$ as well as the $T_2$-radius neighborhoods of every $\dy\in\cB'$. To this end, put
\begin{align*}
 U &\deq \partial B_{T_1}(\dx)~, & &V_i \deq \partial B_{T_2}(\dy_i)~,\\
 \tilde{U} &\deq U \setminus \cup_i B_{T_2}(\dy_i)~, &&\tilde{V}_i \deq V_i \setminus \left(B_{T_1}(\dx)\cup \left(\cup_{j\neq i} B_{T_2}(\dy_j)\right)\right)~.
\end{align*}
Recalling Lemma \ref{lem-dir-root-boundary-sizes} (and the fact that $\{\dx\}\cup \cB'$ are all directed $L$-roots), with probability $1-o(n^{-3})$ we have
\begin{align*}
  |U| &\geq \left(1-O(n^{-\frac15})\right)(d-1)^{T_1}~,\\
  |V_i| &\geq \left(1-O(n^{-\frac15})\right)(d-1)^{T_2}~\mbox{ for all $i\in[M']$~.}
\end{align*}
Combining this with Lemma \ref{lem-dir-k-roots-intersections}, we deduce that
for any sufficiently large $n$ the following holds with probability $1-o(n^{-3})$:
\begin{align*}
\left(1-2n^{-\frac17}\right)(d-1)^{T_1} &\leq |\tilde{U}| \leq (d-1)^{T_1}~,\\
\left(1-2n^{-\frac17}\right)(d-1)^{T_2} &\leq |\tilde{V}_i| \leq (d-1)^{T_2}~\mbox{ for all $i\in[M']$~.}
\end{align*}
We will use a standard Poissonization approach in order to approximate the joint distribution of the variables
$\{Z_\dy : \dy\in\cB'\}$ (that are fully determined by the graph $G$) using the following set of variables:
$$\tilde{Z}_{\dy_i} \deq \left|\{ u,v\in E: u\in \tilde{U},~v\in \tilde{V}_i\}\right|\quad(i\in[M'])~.$$
We claim that $\tilde{Z}_{\dy_i} \leq Z_{\dy_i}$ for all $i$. To see this, recall that $Z_{\dy_i}$ counts the number of non-backtracking paths of length $T+K-1$ from $\dx$ to $\dy_i$. Since $\tilde{U}$ and $\tilde{V}_i$ are disjoint subsets of the boundaries of the $T_1$-radius neighborhood $U$ and the $T_2$-radius neighborhood $V_i$ respectively, every edge between them corresponds to at least one distinct such path of length $T_1+T_2+1=T+K-1$.
Therefore, by the triangle inequality,
\begin{align}
M \E \bigg[\Big|\frac{Z_\cB}{\mu}-1\Big|&\given\cF_G\bigg] = \sum_{\dy\in\cB}  \Big| \frac{Z_{\dy}}{\mu}-1  \Big|\leq  \sum_{i=1}^{M'}  \Big| \frac{Z_{\dy_i}}{\mu}-1  \Big| + \sum_{\dy\in \cB\setminus \cB'}\frac{Z_\dy}{\mu}
+\left|\cB\setminus\cB'\right| \nonumber\\
&\leq \sum_{i=1}^{M'} \Big(\Big|\frac{\tilde{Z}_{\dy_i}}{\mu}-1\Big|+ \frac{Z_{\dy_i}-\tilde{Z}_{\dy_i}}{\mu}\Big)+\sum_{\dy\in \cB\setminus \cB'}\frac{Z_\dy}{\mu}
+3 \nonumber\\
&= \sum_{i=1}^{M'}  \Big|\frac{\tilde{Z}_{\dy_i}}{\mu}-1\Big|- \sum_{i=1}^{M'} \frac{\tilde{Z}_{\dy_i}}{\mu}+\sum_{\dy\in \cB}\frac{ Z_\dy}{\mu}
+3~. \label{eq-E-ZB-bound-1}
\end{align}
Let $\tilde{\mathcal{Z}}$ denote the first summand in the last expression:
$$ \tilde{\mathcal{Z}} \deq \sum_{i=1}^{M'} |(\tilde{Z}_{\dy_i}/\mu)-1|~.$$
The following lemma estimates $\mathcal{Z}$, as well as the second summand in \eqref{eq-E-ZB-bound-1}.
\begin{lemma}\label{lem-cZ-bound}
Define $\tilde{\mathcal{Z}}$ and $\tilde{Z}_{\dy_i}$ for $i=1,\ldots,M'$ as above. Then:
\begin{equation}
  \label{eq-tilde-Z-upper-bound}
  \P\left(\tilde{\mathcal{Z}}  > \epsilon + \mbox{$\frac{4}{\log\log n}$}\right) = o(n^{-2})\,,
\end{equation}
and
\begin{equation}\label{eq-tilde-Z-lower-bound} \P\left(\frac1{M}\sum_{i=1}^{M'}\frac{\tilde{Z}_{\dy_i}}{\mu} \leq 1-\epsilon - \mbox{$\frac{1}{\log\log n}$}\right)
= o(n^{-2})~.
\end{equation}
\end{lemma}
\begin{proof}
We claim that, with probability $1-o(n^{-2})$, each of the variables $\tilde{Z}_{\dy_i}$ is stochastically dominated from below and from above by i.i.d.\ pairs of binomial variables, $R^-_i \leq R^+_i$ (coupled in the obvious manner), defined as:
\begin{align*}
 R^-_i &\sim \Bin\left((1-n^{-\frac18})(d-1)^{T_2+1},p^- \right),&&p^- \deq (1-n^{-\frac18})\frac{(d-1)^{T_1+1}}{dn}~,\\
   R^+_i &\sim \Bin\left((d-1)^{T_2+1},p^+\right)~,&&p^+\deq (1+n^{-\frac14})\frac{(d-1)^{T_1+1}}{dn}~,\\
   \Delta_i &\deq R^+_i - R^-_i \geq 0~.
\end{align*}
To see this, consider the configuration model at the starting phase where the vertices in $\tilde{U}\cup(\cup_i\tilde{V}_i)$
all have degree $1$ (that is, each of these vertices comprise $(d-1)$ points that still wait to be paired), and expose
the pairings of the points in $\tilde{V}_i$ sequentially. Suppose that for all $j<i$ we have already constructed a coupling
where $R^-_j \leq \tilde{Z}_{\dy_j} \leq R^+_j$, and next wish to do the same for $\tilde{Z}_{\dy_i}$.

By Lemma \ref{lem-dir-k-roots-intersections}, with probability $1-o(n^{-3})$ there still remain at least $(1-n^{-1/8})(d-1)^{T_2}$ vertices of degree 1 in $\tilde{V}_i$ and at least $(1-n^{-1/8})(d-1)^{T_1}$ such vertices in $\tilde{U}$ (otherwise
the intersection of either $B(\dy_i)$ or $B(\dx)$ with one of $B(\dy_1),\ldots,B(\dy_{i-1})$ would contain at least $n^{-1/7}(d-1)^{T_1}$ vertices). We thus have at least $(1-n^{-1/8})(d-1)^{T_2+1}$ unmatched points corresponding to $\tilde{V}_i$,
and at least $(1-n^{-1/8})(d-1)^{T_1+1}$ unmatched points corresponding to $\tilde{U}$. Associating each such point corresponding to $\tilde{V}_i$ with a Bernoulli variable, which succeeds if and only if it is matched to $\tilde{U}$, clearly establishes
the coupling of $\tilde{Z}_{\dy_i} \geq R^-_i$.

Conversely, $\tilde{V}_i \leq (d-1)^{T_2}$ and $\tilde{U}\leq (d-1)^{T_1}$, hence there are at most $(d-1)^{T_2+1}$ unmatched points
corresponding to $\tilde{V}_i$ and at most $(d-1)^{T_1+1}$ unmatched points corresponding to $\tilde{U}$. Since both the $T_1$-radius and the $T_2$-radius neighborhoods of any element contains $O(\sqrt{n})$ distinct vertices, the probability of a point corresponding to $\tilde{V}_i$ being matched to $\tilde{U}$ is at most
$$\frac{(d-1)^{T_1+1}}{dn - O(M \sqrt{n})} \leq \frac{(d-1)^{T_1+1}}{(1- o(n^{-1/4}))dn}~.$$
Therefore, we can readily construct the coupling $\tilde{Z}_{\dy_i} \leq R^+_i$.

Since it was possible to construct each of the above couplings with probability $1-o(n^{-3})$, clearly all $M'$ variables can be coupled as above with probability $1-o(n^{-2})$.


Finally, consider a set of i.i.d.\ binomial random variables $Q_i$ with means $\E Q_1 = \mu = (d-1)^{T+K}/dn$, defined by
$$Q_i \sim \Bin\Big((d-1)^{T_2+1},\frac{(d-1)^{T_1+1}}{dn}\Big)~,$$
and coupled in the obvious manner such that $R^-_i \leq Q_i \leq R^+_i$. Clearly, as
$|\tilde{Z}_{\dy_i}- Q_i| \leq R^+_i - R^-_i = \Delta_i$, it follows that
\begin{equation}\label{eq-tilde-Z-bound}
\tilde{\mathcal{Z}}= \frac{1}{M'}\sum_{i=1}^{M'} \Big|\frac{\tilde{Z}_{v_i}}{\mu}-1\Big| \leq
\frac{1}{M'}\sum_{i=1}^{M'} \Big|\frac{Q_i}{\mu}-1\Big| +
\frac{1}{M'}\sum_{i=1}^{M'} \frac{\Delta_i}{\mu}~.
\end{equation}
Since $\mu\geq  (d-1)^{K} \geq 1/\epsilon^2$, for all $i\in[M']$ we have
\begin{align*}
&\E \Big| \frac{Q_i}{\mu}-1\Big| \leq \frac{1}{\mu}\sqrt{\var(Q_i)} = \frac{1+O(n^{-\frac14})}{\sqrt{\mu}}
\leq \left(1+\frac1{\log n}\right)\epsilon~,\\
&\frac{\E\Delta_i}{\mu} \leq (1-n^{-\frac1{8}})\left( n^{-\frac14}+n^{-\frac{1}{8}}\right) + n^{-\frac{1}{8}}\left(1+n^{-\frac14}\right) = O\left(n^{-\frac14}\right)~.
\end{align*}
where the last inequalities in both estimates hold for any sufficiently large $n$.
Furthermore, since the $\{Q_i\}$-s are i.i.d.\ binomial variables, Chernoff's inequality implies that
\begin{align}
\P\Big( \frac{1}{M'}\sum_{i=1}^{M'} \frac{Q_i}{\mu} > 1+\mbox{$\frac1{\log\log n}$}\Big) &< \mathrm{e}^{- \frac{\mu M'}{4(\log\log n)^2}}= \mathrm{e}^{-\Omega\left((\frac{\log n}{\log\log n})^2\right)} = o(n^{-2})~,\label{eq-Qi-bound}
\end{align}
and an analogous argument for the $\{\Delta_i\}$-s (recall that by definition, we have $\Delta_i = \Delta_i'+\Delta_i''$, where the $\{\Delta_i'\}$-s and $\{\Delta_i''\}$-s are two sequences of i.i.d.\ binomial variables, independent of each other), combined with the fact that $\E \Delta_i / \mu = O\left(n^{-1/4}\right)$, gives
\begin{equation}   \label{eq-Deltai-bound}  \P\Big(\frac1{M'}\sum_{i=1}^{M'}\frac{\Delta_i}{\mu}  > \mbox{$\frac1{\log\log n}$}\Big)
  \leq \mathrm{e}^{-\Omega\left((\frac{\log n}{\log\log n})^2\right)}
  = o(n^{-2})~.
\end{equation}
Define
$$X_t \deq \sum_{i=1}^t\Big|\frac{Q_i}{\mu}-1\Big|-\Big(\frac{Q_i}{\mu}-1\Big)-\E\Big|\frac{Q_i}{\mu}-1\Big|~.$$
Since $\E\left|(Q_i/\mu)-1\right| \leq (1+\frac{1}{\log n})\epsilon < 2$ for large $n$ (with room to spare), we deduce that $X_t$ is a martingale with bounded increments:
$$\left|X_{t+1}-X_t\right| \leq 2 + \E\Big|\frac{Q_i}{\mu}-1\Big| \leq 4~.$$
Therefore, Azuma's inequality (cf., e.g., \cite{AS}*{Chapter 7.2}) implies that
\begin{equation}
  \label{eq-Xi-martingale-bound}
  \P\left(X_{M'} / M' > \mbox{$\frac1{\log\log n}$}\right) < \mathrm{e}^{-\frac12 M'/(4\log\log n)^2} = o(n^{-2})~.
\end{equation}
Since $\E\left|(Q_1/\mu)-1\right| < (1+\frac{1}{\log n})\epsilon$ and
$$ \frac{1}{M'}\sum_{i=1}^{M'}\Big|\frac{Q_i}{\mu}-1\Big| = \E\Big|\frac{Q_1}{\mu}-1\Big| + (X_{M'}/M')+ \frac{1}{M'}\sum_{i=1}^{M'}\Big(\frac{Q_i}{\mu}-1\Big)~,$$
the bounds in \eqref{eq-Qi-bound} and \eqref{eq-Xi-martingale-bound} now imply that
\begin{align*}
\P\Big(\frac1{M'}\sum_{i=1}^{M'}\Big| \frac{Q_i}{\mu}-1\Big| > \epsilon+\mbox{$\frac3{\log\log n}$}\Big)= o(n^{-2})~.
\end{align*}
Together with \eqref{eq-tilde-Z-bound} and \eqref{eq-Deltai-bound}, we obtain that \eqref{eq-tilde-Z-upper-bound} indeed holds.

Similarly, since $\tilde{Z}_{\dy_i} \geq R^-_i$ for all $i$, and the $\{R^-_i\}$-s are i.i.d.\ binomial variables
with $\E R^-_i \geq (1-\epsilon-3n^{-1/8})\mu$, we can apply Chernoff's inequality to derive a lower bound on $\sum_{i=1}^{M'} (\tilde{Z}_{\dy_i}/\mu)$. Keeping in mind that
$$\frac1M \sum_{i=1}^{M'}\frac{\tilde{Z}_{\dy_i}}{\mu} \geq \left(1-\frac3M\right)\sum_{i=1}^{M'}\frac{\tilde{Z}_{\dy_i}}{\mu} ~,$$
we obtain that \eqref{eq-tilde-Z-lower-bound} holds, as
\begin{equation*}
\P\left(\frac1{M}\sum_{i=1}^{M'}\frac{\tilde{Z}_{\dy_i}}{\mu} \leq 1-\epsilon - \mbox{$\frac{1}{\log\log n}$}\right)
\leq \mathrm{e}^{-\Omega\left((\frac{\log n}{\log\log n})^2\right)} = o(n^{-2})~.
\end{equation*}
This completes the proof of Lemma~\ref{lem-cZ-bound}.
\end{proof}
We can now combine \eqref{eq-tilde-Z-upper-bound} and \eqref{eq-tilde-Z-lower-bound} with \eqref{eq-E-ZB-bound-1},
and deduce that the following statement holds with probability $1-o(n^{-2})$:
\begin{align}
\E \bigg[\Big|\frac{Z_\cB}{\mu}-1\Big|&\given\cF_G\bigg] \leq 2\epsilon - 1 + \mbox{$\frac{5}{\log\log n}$}
+\frac1M\sum_{\dy\in \cB}\frac{ Z_\dy}{\mu} ~. \label{eq-E-ZB-bound-2}
\end{align}
To transform the above into the required bound on $Z$, take $M = \lceil \log^2 n \rceil$,
and consider a collection of bins, each of size either $M$ or $M+1$, such that the total of their sizes is $d n$.
Let $\cB'_{1},\ldots,\cB'_{\ell_1}$ denote the $M$-element bins, and let $\cB''_{1},\ldots,\cB''_{\ell_2}$
denote the $(M+1)$-element bins. Next, randomly partition the elements of $\vE$ into these bins (i.e.,
each bin $\cB$ will contain a uniformly chosen set of $|\cB|$ directed edges).

Since there are at most $\lfloor dn/M\rfloor = O(n/M)$ different bins, and for each bin the corresponding $Z_{\cB}$
satisfies \eqref{eq-E-ZB-bound-2} with probability at least $1-o(n^{-2})$, we deduce that all the variables $Z_{\cB'_{j}}$
and $Z_{\cB''_{j}}$ satisfy this
inequality with probability at least $1-o(1/n)$. Therefore, with probability at least $1-o(1/n)$,
\begin{align*}
  \E \bigg[\Big| \frac{Z}{\mu} - 1 \Big|&\given\cF_G\bigg] =
  \frac{1}{dn}\sum_{\dy\in \vE}\Big|\frac{Z_\dy}{\mu}-1\Big| \\
  &= \frac{M}{dn}\sum_{j=1}^{\ell_1}\E\bigg[\Big| \frac{Z_{\cB'_j}}{\mu} - 1 \Big|\given\cF_G\bigg]
 + \frac{M+1}{dn}\sum_{j=1}^{\ell_2}\E\bigg[\Big| \frac{Z_{\cB''_j}}{\mu} - 1 \Big|\given\cF_G\bigg] \\
&\leq 2\epsilon -1 +  \mbox{$\frac5{\log\log n}$} +\frac1{dn}\sum_{\dy \in \vE} \frac{Z_{\dy}}{\mu}  = 2\epsilon + \mbox{$\frac5{\log\log n}$}~,
\end{align*}
where the last equality follows from the fact that $$\sum_\dy Z_\dy = \sum_\dy \cP_{T+K-1}(\dx,\dy) = (d-1)^{T+K} = \mu dn~.$$
This completes the proof.
\end{proof}

\begin{proof}[\emph{\textbf{Proof of Theorem \ref{thm-nbrw}}}]
Let $(\vW_t)$ be the non-backtracking random walk, and let $\pi$ denote the stationary
distribution on $\vE$.

The lower bound is a consequence of the following simple claim:
\begin{claim}\label{cl:mixingLowerBound}
Every $d$-regular graph on $n$ vertices satisfies
$$\tmix(1-\epsilon) \geq \lceil\log_{d-1}(dn)\rceil - \lceil\log_{d-1}(1/\epsilon)\rceil~\mbox{ for any $0<\epsilon < 1$}~.$$
\end{claim}
\begin{proof}[Proof of claim]
Let $\epsilon > 0$ and let $\dx_0\in\vE$ be any starting position.
 Clearly, at time $T=\lfloor\log_{d-1}(\epsilon dn)\rfloor$ we have
$$|\partial B_{T}(\dx_0)| \leq (d-1)^{T} \leq \epsilon dn~,$$ and the set $A\deq \vE\setminus \partial B_{T}(\dx_0)$ has stationary measure at least $1- \epsilon$. Thus,
$$\|\P_{\dx_0}(\vW_{T}\in\cdot)-\pi\|_\mathrm{TV} \geq \left|\P_{\dx_0}(\vW_{T}\in A)-\pi(A)\right| \geq 1-\epsilon~,$$
implying that $\tmix(1-\epsilon) > T$. The proof now follows from the fact that
\begin{align*}
\lceil\log_{d-1}(dn)\rceil - \lceil\log_{d-1}(1/\epsilon)\rceil &= \lceil\log_{d-1}(dn)\rceil + \lfloor\log_{d-1}\epsilon\rfloor \\
&\leq \lceil \log_{d-1}(\epsilon dn)\rceil \leq T+1~.\qedhere
\end{align*}
\end{proof}
For the upper bound, let $\dx_0$ be the worst starting position, and let $\dx= \vW_{t_0}$, where
$t_0 = \lceil \log_{d-1}(2/\epsilon)\rceil$. Let $\droot$
denote the event that $\dx$ is a directed $L$-root, where $L=\lceil\frac16\log_{d-1}(dn)\rceil$.
Conditioning on the statements of Lemma \ref{lem-dir-leq-one-bad-edge}
and Lemma \ref{lem-nbrw-to-good-roots} (and recalling that both hold \whp) we obtain that $ \P_{\dx_0} (\droot) \geq 1-\epsilon$.

Condition on the statement of Proposition \ref{prop-poisson}, and following its notation,
let $Z(\dx)$ accept the value $\cP_{T+K-1}(\dx,\dy)$ with
probability $1/dn$, where
$$ K = \lceil2\log_{d-1}(1/\epsilon)\rceil~, \quad T = \lceil\log_{d-1} (dn)\rceil ~, \quad\mu = (d-1)^{T+K}/dn~.$$
The following then holds:
\begin{align}
\sum_{\dy \in \vE}& \left|\P_{\dx}(\vW_{T + K} = \dy\mid \droot) - \frac{1}{dn}\right|\nonumber\\
&= \sum_{k} \left|\left\{\dy: \cP_{T+K-1}(\dx,\dy)=k \right\}\right|
\left|\frac{k}{(d-1)^{T+K}} - \frac{1}{dn}\right|\nonumber\\
&= \sum_{k} \P\left(Z = k\mid \cF_G\right) \left|\frac{k}{\mu} - 1\right| =  \E \left[\left| (Z/\mu) - 1\right|\given\cF_G\right] \leq 2\epsilon+o(1)~,\label{eq-tv-given-LR}
\end{align}
where in the last inequality we applied Proposition \ref{prop-poisson} onto the directed $L$-root $\dx$ (given the event $\droot$).
We deduce that for $t(\epsilon) = t_0 + T + K$:
\begin{align} \Big\|\P_{\dx_0}&(\vW_{t}\in\cdot)-\pi\Big\|_\mathrm{TV}
= \frac12 \sum_{\dy \in \vE} \left|\P_{\dx_0}(\vW_{t} = \dy) - \frac{1}{dn}\right|
 \nonumber \\&\leq
 \frac12 \P_{\dx_0}(\droot) \sum_{\dy \in \vE} \left|\P_{\dx_0}(\vW_{t} = \dy\mid \droot) - \frac{1}{dn}\right|
 + \P_{\dx_0}(\droot^c)\nonumber\\
 &\leq \epsilon + (1-\epsilon)\P_{\dx_0}(\droot^c) + o(1) \leq 2\epsilon - \epsilon^2 + o(1) < 2\epsilon~,\label{eq-tv-upper}
 \end{align}
where the first inequality in the last line is by \eqref{eq-tv-given-LR}, the second one is due to
the fact that $\P(\droot^c)\leq\epsilon$, and the third inequality holds for sufficiently large values of $n$.
Therefore, for any large $n$ we have
\begin{align*}
\tmix(\epsilon) &\leq t(\epsilon/2) \leq \lceil\log_{d-1}(dn)\rceil + 3\left\lceil \log_{d-1}(2/\epsilon)\right\rceil+ \lceil\log_{d-1}2\rceil\\
&\leq \lceil\log_{d-1}(dn)\rceil + 3\lceil \log_{d-1}(1/\epsilon)\rceil+ 4
\end{align*}
(where in the last inequality we used the fact that $d\geq 3$), as required.
\end{proof}

\section{Cutoff for random regular graphs of large degree}\label{sec:large-d}

In this section, we prove Theorem \ref{thm-nbrw-large-d} and Corollary \ref{cor-rw-large-d}, which extend our cutoff result for the \SRW\ and \NBRW\ on almost every random regular graph of fixed degree $d\geq 3$ to the case of $d$ large. To prove cutoff for the \NBRW, we adapt our original arguments (from the case of $d$ fixed) to the new delicate setting
where our error probabilities are required to be exponentially small in $d$. The behavior of the \SRW\ is then obtained as a corollary of this result.

Throughout the section, let
$d=d(n) \to\infty$ with $n$, and recall that we further assume that $d = n^{o(1)}$, since otherwise the the mixing time is $O(1)$ and cutoff is impossible.

\subsection{NBRWs on random regular graphs of large degree}

As we will soon show, when $d$ is large we no longer need to deal with $K$-roots (and the locally-tree-like geometry of the starting point of our walk), as all vertices will have sufficient expansion \whp.  However, the analysis of the configuration model becomes more delicate, as the probability that it produces a simple graph is $(1+o(1))\exp\big(\frac{1-d^2}4\big)$ (see \eqref{eq-p-simple}), which now decays with $n$. Thus, to prove that the probability of an event goes to 0 on $\G(n,d)$, we must now show that its probability is $o\left(\exp(-d^2/4)\right)$ in the configuration model.

\begin{lemma}\label{l:largeDBoundarySize}
With high probability, for all $\dx\in\vE$ and all $t \leq \frac47\log_{d-1}n$,
\begin{equation}\label{e:largeDEdgeBound}
 |\partial B_t(\dx)| \geq (1-o(1))(d-1)^t ~.
 \end{equation}
\end{lemma}
\begin{proof}

The proof is an adaption of Lemma \ref{lem-u-boundary-sizes}.  Pick a directed edge $\dx$ uniformly at random and expose its first level.  Since we are interested in probabilities conditioned on the graph $G$ being simple, we may assume that $|\partial B_1(\dx)|=d-1$, that is, there are no self-loops or multiple edges from $\dx$.

We will show that \eqref{e:largeDEdgeBound} holds with probability $1-o\left(n^{-1} \exp(-d^2/4)\right)$ for the above $\dx$ in the configuration model.  Clearly, for any $t<t'$ we have  $|\partial B_t(\dx)| \geq  (d-1)^{t'-t}|\partial B_{t'}(\dx)|$, hence we can restrict our attention to $\partial B_T(\dx)$ where $T = \lfloor\frac47\log_{d-1}n\rfloor$.

Following the notation in the proof of Lemma \ref{lem-leq-one-bad-edge}, let $A_{i,k}$ be the event that, in the process of sequentially matching points, the newly exposed pair of the $k$-th unmatched point in $\partial B_i$ belongs to some vertex already in $B_{i+1}$.  Further recall that, by \eqref{eq-srw-bad-edge-i-k} and the discussion thereafter,
the number of events $\{A_{i,k} : 0\leq i < T \}$ that occur is stochastically dominated by
a binomial variable with parameters $\Bin\left((d-1)^{T+1}, \frac{(d-1)^{T}}{n}\right)$. By our choice of $T$, the expectation of this random variable is
$$ (d-1)^{2T+1} / n \leq d n^{1/7} \leq n^{1/7+o(1)}~,$$ hence the number of events $A_{i,k}$ with $0\leq i < T$ that occur is less than $n^{1/6}$ (with room to spare) with probability at least $1 - \exp(-\Omega(n^{1/6}))$.
Next, set $$L = \left\lfloor \tfrac15\log_{d-1} n\right\rfloor,\qquad \rho=\left\lceil4+2d^2/\log n\right\rceil=o(d^2).$$  As before, we can stochastically dominate the number of events $A_{i,k}$ that occur in the first $L$ levels, $\{A_{i,k} : 0\leq i < L \}$, by
a binomial variable $X_L \sim \Bin\left((d-1)^{L+1}, \frac{(d-1)^{L}}{n}\right)$. Since the expected value of $X_L$ is
$$(d-1)^{2L+1} / n = o\big(n^{-1/2}\big)~,$$
and since $L\to\infty$ with $n$ (by our assumption on $d$), it is easy to verify that
$$ \P(X_L \geq \rho) = (1+o(1))\P(X_L = \rho) = o\big(n^{-\rho/2}\big)~.$$
Recalling the definition of $\rho$, it now follows that the number of events $A_{i,k}$ with $0\leq i < L$ that occur is less than $\rho$ except with probability $o(n^{-2} \mathrm{e}^{-d^2})$.

Each event $A_{i,k}$ reduces the number of leaves in level $i+1$ by at most 2, hence it reduces the number of leaves in level $t>i$ by at most $2(d-1)^{t-i-1}$ vertices. It then follows that for each $0\leq t < T$,
\begin{equation}\label{e:badVertexBound-1-large-d}
|\partial B_t(\dx)| \geq (d-1)^{t} - \sum_{i<t}\sum_k \one_{A_{i,k}} 2(d-1)^{t-i-1}~.
\end{equation}
As $|\partial B_1(\dx)|=d-1$, there are no events of the form $A_{0,k}$. Therefore, by the discussion above, with probability $1-o(n^{-2} \mathrm{e}^{-d^2})$ we have
$$ \sum_{i < L} \sum_k \one_{A_{i,k}} 2(d-1)^{t-i-1} \leq
2(d-1)^{t-2}\rho = o\left((d-1)^{t}\right).$$
Furthermore, by the above discussion on the number of events $A_{i,k}$ that occur, we deduce that
with probability at least $1-\exp(-\Omega(n^{1/6}))$
$$ \sum_{i =L}^{t-1} \sum_k \one_{A_{i,k}} 2(d-1)^{t-i-1} \leq
2(d-1)^{t-L-1} n^{1/6} = o\left((d-1)^t\right).$$
Plugging the above in \eqref{e:badVertexBound-1-large-d}, we obtain that with probability $1-o(n^{-2} \mathrm{e}^{-d^2})$
\begin{equation}\label{e:badVertexBound-2-large-d}
|\partial B_t(\dx)| \geq (1-o(1))(d-1)^{t}~,
\end{equation}
and a union bound implies that \eqref{e:badVertexBound-2-large-d} holds for all directed edges $\dx$ which satisfy $|\partial B_1(\dx)|=d-1$ except with probability $O\big(\frac{d}n\exp(-d^2)\big)=o(\exp(-d^2))$. By \eqref{eq-p-simple},
it now follows that \eqref{e:largeDEdgeBound} also holds \whp\ over $\G(n,d)$.
\end{proof}

The following lemma, the analogue of Lemma \ref{lem-k-roots-non-intersections}, is proved by essentially following the same argument as in the proof of  Lemma \ref{lem-k-roots-non-intersections}, i.e., calculating the size of the common neighborhood of two vertices.
 The difference is again that here we need to deal with the fact that the probability that the configuration model is a simple graph is exponentially small in $d$. This is achieved by repeating the approach, demonstrated in Lemma \ref{l:largeDBoundarySize} above, of treating $B_1(\dx)$ separately. Applying this analysis to the neighborhoods of the 2 starting directed edges $\dx,\dy$ gives the required result, with the remaining arguments of Lemma \ref{lem-k-roots-non-intersections} left unchanged (we omit the full details).
\begin{lemma}
Set $T = \frac{51}{100}\log_{d-1}n$ and $L = \frac16\log_{d-1}n$.
Then \whp, for every $\dx,\dy\in\vE$ with $\dist(\dx,\dy) > 2L$ and every $t \leq T$,
$$ | B_t(\dx)\cup B_t(\dy) | \geq n^{-1/7}(d-1)^t ~.$$
\end{lemma}

The final ingredient needed is the analogue of the Poissonization result of Proposition \ref{prop-poisson},
as given by the following proposition.

\begin{proposition}\label{prop-poisson-large-d}
Let $\epsilon > 0$, set
\begin{align*}
T  = \lceil \log_{d-1}(dn) + 2\log_{d-1}(1/\epsilon) \rceil~, \quad\mu = (d-1)^{T}/dn~,
\end{align*}
and for each $\dx\in\vE$, define the random variable $Z=Z(\dx)$ by
\begin{align*}
  \P(Z = k) = \frac{1}{dn}\left|\left\{ \dy \in \vE : \cP_{T-1}(\dx,\dy) =k \right\} \right|~,
\end{align*}
where $\cP_\ell(\dx,\dy)$ is the number of $\ell$-long non-backtracking paths from $\dx$ to $\dy$.
Then \whp, every $\dx$  satisfies
$$\E \left[\left| (Z(\dx)/\mu)-1\right|\given \cF_G \right]<2\epsilon+\frac{5}{\log\log n}~,$$
where $\cF_G$ is the $\sigma$-field generated by the graph $G\sim \mathcal{G}(n,d)$.
\end{proposition}
The proof of the above proposition is essentially the same as the proof of Proposition \ref{prop-poisson},
with some minor adjustments to the estimates to ensure that they hold with probability $o\left(\exp(-d^2/4)\right)$.  The main necessary change is to let the bin sizes depend on $d$, namely to set $M=d^3 \log^2 n$. As only minor adjustments to some of the bounds are required elsewhere, we omit the details.

\begin{proof}[\emph{\textbf{Proof of Theorem \ref{thm-nbrw-large-d}}}]
The lower bound of $\tmix(s) \geq \lceil \log_{d-1} (dn) \rceil$ follows immediately from Claim \ref{cl:mixingLowerBound}, whose proof remains valid without change, even when $d$ is allowed to grow with $n$.

To obtain the upper bound, let $(\vW_t)$ denote the non-backtracking random walk started at $\vW_0 = \dx$. Set $\epsilon=3s$, and $$T  = \lceil \log_{d-1}(dn) + 2\log_{d-1}(1/\epsilon) \rceil~, \quad\mu = (d-1)^{T}/dn~.$$  By Proposition \ref{prop-poisson-large-d} we have that \whp,
\begin{align*}
\sum_{\dy \in \vE}& \left|\P_{\dx}(\vW_{T } = \dy) - \frac{1}{dn}\right|\nonumber\\
&= \sum_{k} \left|\left\{\dy: \cP_{T-1}(\dx,\dy)=k \right\}\right|
\left|\frac{k}{(d-1)^{T}} - \frac{1}{dn}\right|\nonumber\\
&= \sum_{k} \P\left(Z = k\mid \cF_G\right) \left|\frac{k}{\mu} - 1\right| \\
&=  \E \left[\left| (Z/\mu) - 1\right|\given\cF_G\right] \leq 2\epsilon+o(1)\leq s
\end{align*}
for large $n$. We conclude that $\tmix(s) \leq T \leq \lceil \log_{d-1} (dn) \rceil + 1$, since $\log_{d-1}(1/\epsilon) =o(1)$ by our assumption on $d$.
\end{proof}

\subsection{Duality between non-backtracking and simple random walks}
The following observation is attributed to Yuval Peres:
\begin{observation}\label{obs:Yuval}
  \label{obs-cover-tree}
  Conditioning on being in level $k$ of the simple random walk on the tree, we are uniform over $k$-long non-backtracking
  random walks.
\end{observation}

More specifically, let $\cT$ be the cover tree for $G$ at $u$ with a map $\phi$, as defined in \eqref{eq-cover-tree}.  Let $X_t$ be a \SRW\ on $\cT$ started from $\rho$ and let $W_t=\phi(X_t)$ be the corresponding \SRW\ on $G$ started from $u$.  Compare this with a \NBRW\ random walk $\vW_t$ started from $\dx=(w,u)$ where $w$ is chosen uniformly from the neighbors of $u$.  For a directed edge $(y,z)$ let $\psi(\cdot)$ denote the projection $\psi((y,z))=z$, giving the vertex the \NBRW\ is presently situated at.

Note that, by symmetry, conditioned on $\dist(\rho,X_t)=k$ the random walk is uniform on the $d(d-1)^{k-1}$ points $\{w \in \cT: \dist(\rho,w)=k\}$.  By the obvious one-to-one correspondence between paths of length $k$ from $\rho$ in $\cT$ and non-backtracking paths of length $k$ in $G$ from $u$, the following holds: Conditioned on $\dist(\rho,X_t)=k$ we have that $W_t$ is distributed as $\psi(\vW_k)$.  Thus, if $\vW_t$ is mixed at time $k$ then a \SRW\ will be mixed once its lift to the cover tree reaches distance $k$ from the root.

\begin{proof}[\emph{\textbf{Proof of Corollary \ref{cor-rw-large-d}}}]

In our proof of  Theorem \ref{thm-rw}, it was shown using the Central Limit Theorem (see equation \eqref{e:clt})
that the distance from the root of the walk in the cover tree is given by
\begin{equation}\label{e:clt2}
\frac{\dist(X_{t},\rho) - \frac{(d-2)t}{d}}{\frac{2\sqrt{d-1}}{d}\sqrt{ t}} \stackrel{\mathrm{d}}{\longrightarrow} N(0,1).
\end{equation}
When $d$ grows with $n$ this Gaussian approximation still holds provided the variance satisfies $\frac{2\sqrt{d-1}}{d}\sqrt{ t}\rightarrow\infty$ or equivalently $(t/d)\to \infty$.  When $d$ and $t$ are of the same order, the number of backtracking steps is asymptotically a Poisson random variable with mean $(t/d)$, therefore $\left(t-\dist(X_{t},\rho)\right)$ is distributed as twice a $\Po(t/d)$ random variable. In both of these cases, whenever $t$ has order $\log_{d-1}n$, the variance of $\dist(X_{t},\rho)$ is of order $\frac{\log n}{d\log d}$. Finally, when $t/d \to 0$, the number of backtracking steps goes to 0 as well.
This understanding of $\dist(X_t,\rho)$ will allow us to translate our results on \NBRW s into statements on \SRW s.

If $w\in \cT$ and $\dist(\rho,w)\leq R$ then $\phi(w)\in B_R$ and hence,
\begin{align*}
\| \P(W_t\in \cdot)-\pi \|_{\mathrm{TV}} &\geq \P(W_t\in B_R) - \pi(B_R) \geq \P(\dist(X_t,\rho) \leq R) -\pi(B_R).
\end{align*}
In particular, as $|B_R|\leq O\big(\frac{n}{d-1}\big) = o(1)$ for $R\leq \log_{d-1} (n) - 1$, we have that
\begin{align}\label{e:tvLBound-large-d}
\| \P(W_t\in \cdot)-\pi \|_{\mathrm{TV}} &\geq \P\left(\dist(X_t,\rho) \leq \log_{d-1}(n) - 1\right) -o(1).
\end{align}

Next, let $\varrho_k=d_{\mathrm{TV}}(\vW_k,\pi)$
be the total variation distance between the \NBRW\ and the stationary distribution.  According to Observation \ref{obs:Yuval} (the correspondence between walks on the cover tree conditioned to be at distance $k$ and \NBRW s of length $k$), the following holds:
\begin{align*}
\| \P(W_t\in \cdot)-\pi \|_{\mathrm{TV}} &\leq  \sum_{k=0}^t \| \P\left (W_t\in \cdot \mid \dist(X_t,\rho) = k \right)-\pi \|_{\mathrm{TV}} \nonumber\\
&~~\qquad\cdot\P(\dist(X_t,\rho) = k)\\
&=  \sum_{k=0}^t  \varrho_k \P(\dist(X_t,\rho) = k)~.
\end{align*}
Now, by Theorem \ref{thm-nbrw-large-d}, when $k >  \lceil \log_{d-1}(dn) \rceil$ we have $\varrho_k = o(1)$, hence
\begin{align}\label{e:tvUBound-large-d}
\| \P(W_t\in \cdot)-\pi \|_{\mathrm{TV}} \leq \P\left(\dist(X_t,\rho) \leq  \lceil \log_{d-1}(dn) \rceil \right) + o(1).
\end{align}

Equations \eqref{e:tvLBound-large-d} and \eqref{e:tvUBound-large-d} imply that mixing takes place when $\dist(X_t,\rho)$ is $\log_{d-1}n + O(1)$. By the above discussion on the distribution of $\dist(X_t,\rho)$ this occurs when $t$ is around $\frac{d}{d-2}\log_{d-1}n$ with window $\sqrt{\frac{\log n}{d\log d}}$.

It remains to address the case where $\frac{d\log\log n}{\log n}\to\infty$. Notice that here, as the probability of the \SRW\ on $G$ making a backtracking step is $1/d$, the probability of backtracking anywhere in its first $\lceil \log_{d-1} (dn) \rceil+1$ steps is $o(1)$.  Hence, we can couple the \SRW\ and \NBRW\ in their first $\lceil \log_{d-1} (dn) \rceil+1$ steps \whp, implying they have the same mixing time. In particular, we may conclude that
for any fixed $0 < s < 1$,
the worst case total-variation mixing time of the \SRW\ on $G$ \whp\ satisfies
\begin{align*}
\tmix(s) \in \left\{ \lceil \log_{d-1} (dn) \rceil, \lceil \log_{d-1} (dn) \rceil+1\right\} ~,
\end{align*}
as required.
\end{proof}

\section{Concluding remarks and open problems}\label{sec::conclusion}
\begin{list}{\labelitemi}{\leftmargin=2em}
 \item We have established the cutoff phenomenon for \SRW s and \NBRW s on almost every $d$-regular graph on $n$ vertices, where $3 \leq d \leq n^{o(1)}$ (beyond which the mixing time is $O(1)$ and we cannot have cutoff). For both walks, we obtained the precise cutoff location and window:
\begin{enumerate}[1.]
 \item For the \SRW, the cutoff point is \whp\ at $\frac{d}{d-2}\log_{d-1}n$, and in fact, we obtained the \emph{two} leading order terms of $\tmix(s)$ for any fixed $s$.
\item  For the \NBRW, cutoff occurs at $\log_{d-1}(dn)$ \whp\ ($\frac{d}{d-2}$ times faster than the \SRW) with an $O(1)$ window. Moreover, for large $d$, the entire mixing transition takes place within a 2-step cutoff window.
\end{enumerate}

 \item Given our discussion in Section \ref{sec:intro} on expander graphs (and the product-criterion for cutoff),
 it would be interesting to extend our results to any arbitrary family of expanders. While one may design such graphs where
 the \SRW\ has no cutoff, such constructions seem highly asymmetric, and the following conjecture seems plausible (see also \cite{DLP}*{Question 5.2}):
\begin{conjecture}
  The \emph{\SRW} on any family of vertex-transitive expander graphs exhibits cutoff.
\end{conjecture}
\item Similarly, recalling the above comparison of $\tmix$ of the \SRW\ and the \NBRW\ on random regular graphs, it would be interesting
to extend this result to any family of vertex-transitive expander graphs.
\end{list}

\section*{Acknowledgments}

We are grateful to Yuval Peres for sharing his deep insight on cutoff with us over many fruitful discussions, and in particular, for simplifying our proof of Theorem \ref{thm-rw} via the duality between \NBRW s and \SRW s (Observation \ref{obs-cover-tree}). We thank Laurent Saloff-Coste for useful discussions, as well as Jian Ding for comments on an earlier version. We also wish to thank the anonymous referees for helpful comments.


\end{document}